\documentclass [11pt]{article}
\usepackage{a4}
\usepackage{paralist}
\usepackage{graphicx}
\usepackage{amssymb}
\usepackage{float}
\usepackage{amsmath}
\usepackage{psfrag}
\usepackage{amsmath,amsthm}
\usepackage{cleveref}

\usepackage{amsmath,amsthm}

\theoremstyle{plain}
\newtheorem{thm}{Theorem}[section]

\newtheorem{lem}[thm]{Lemma}
\newtheorem{prop}[thm]{Proposition}
\newtheorem{cor}[thm]{Corollary}

\theoremstyle{definition}
\newtheorem{defn}[thm]{Definition}
\newtheorem{ex}[thm]{Example}

\theoremstyle{remark}
\newtheorem{rem}[thm]{Remark}
\newtheorem*{note}{Note}

\title{\textbf{Residual deficiency}}
\author{Mariano Zer\'{o}n-Medina Laris}

\begin{document}

\maketitle

\begin{abstract}
We introduce a new real valued invariant for finitely presented groups called residual deficiency. Its main property is the following. Let $G$ be a finitely presented group. If the residual deficiency of $G$ is greater than one, then $G$ has a finite index subgroup with deficiency greater than one. The latter property is strong. For instance, such a group is large (\cite{baums-pride}), has positive rank gradient and positive first $L^2$-Betti number, among other properties.

We also compute the residual deficiency of some well known families of presentations, prove that residual deficiency minus one is supermultiplicative with respect to finite index normal subgroups and find lower bounds for the residual deficiency of quotients.
\end{abstract}

\section{Introduction}

We say a group $G$ is large if it contains a finite index subgroup which surjects onto the non-abelian free group of rank $2$. Such property, introduced in \cite{pride large}, is invariant under finite index subgroups, finite index supergroups and prequotients. Being large is a strong property, it implies that the group in question contains a non-abelian free subgroup \cite{neumann}, that it is $SQ$-universal (\cite{pride large}) (every countable group is a subgroup of a quotient of $G$), it has subgroups with arbitrarily large first Betti number \cite{lubotzky}, has uniformly exponential word growth \cite{harpe}, as well as subgroup growth of strict type $n^{n}$ \cite{lubotzky segal}, among other properties.
\vskip 2mm

In this paper, we introduce a new real valued invariant for finitely presented groups which we call \emph{residual deficiency}. Let $G$ be a finitely presented group, then the main property of the residual deficiency of $G$ is that it helps construct a lower bound for the deficiency of some finite index subgroups of $G$. This is useful because it helps detect largeness and study the class of finitely presented large groups, as we now explain.

\vskip 2mm

Deficiency and largeness were first related in \cite{baums-pride} by B. Baumslag and S. Pride. The result says that any group with deficiency greater than one is large. Largeness is invariant under finite index supergroups. Hence, if $G$ is a group with a finite index subgroup $H$ which has deficiency greater than one, then $G$ is also large. Therefore, looking at the deficiency of finite index subgroups of a group $G$ can help detect largeness. 
\vskip 2mm

Having a finite index subgroup with deficiency greater than one is much stronger than being large. By recent results of I. Agol (\cite{agol}), the fundamental group of a closed, orientable, hyperbolic $3$-dimensional manifold is large. By results in \cite{epstein}, we know that these groups, along with all its finite index subgroups have deficiency zero. This gives a substantial collection of finitely presented groups which are large but do not have a finite index subgroup with deficiency greater than one. Therefore, the latter property is only enjoyed by some large groups and hence can be used to study the class of large finitely presented groups.
\vskip 2mm

In Section $2$ we introduce real valued submultiplicative and supermultiplicative invariants. A real valued invariant $I$, submultiplicative with respect to finite index subgroups, is one such that given a finite index subgroup $H$ in $G$ satisfies
\begin{equation}\label{eq: submult first}
\dfrac{I(H)}{|G:H|}\leq I(G).
\end{equation}
Given an invariant $I$ as above, C.T.C. Wall considered the following gradient (\cite{wall})
\begin{equation}\label{eq: gradient first}
\tilde{I}(G)=\underset{H\underset{f}{\leqslant}G }{\text{inf}} \left\{ \dfrac{I(H)}{|G:H|} \right\}.
\end{equation}
He remarked that such invariant satisfies
\begin{equation}\label{eq: gen euler}
\tilde{I}(G)=\dfrac{\tilde{I}(H)}{|G:H|},
\end{equation}
where $H$ is a finite index subgroup of $G$, though he did not prove it. This property characterises what he defined in \cite{wall} as generalised Euler characteristics. If the inequality in \cref{eq: submult first} is inverted, we obtain a supermultiplicative invariant with respect to finite index subgroups. In this case, a gradient is defined by taking the supremum rather than the infimum.
\vskip 2mm

Many examples of submultiplicative and supermultiplicative invariants can be found in the literature. For instance, consider the rank of a group $G$, which we denote by $d(G)$. Subtracting $1$ from $d(G)$ gives a submultiplicative invariant. Its associated gradient is called the \emph{rank gradient} which we denote by $RG(G)$. Subtracting $1$ from the deficiency of a finitely presented group $G$ is an example of a supermultiplicative invariant. We call the gradient obtained from this the \emph{deficiency gradient} which we denote by $DG(G)$. In Section $2$ we prove the equality of \cref{eq: gen euler}, give an overview of some of these invariants and says how they relate to one another.
\vskip 2mm

Section $3$ starts by defining the residual deficiency of a finitely presented group. We prove its main property which is that it helps construct a lower bound for the deficiency of some finite index subgroups of the group. We then relate it to the gradients studied in Section $2$ and compute the residual deficiency of some well known families of finite presentations. For example, families of presentations which define generalised triangle groups, Coxeter groups and generalised tetrahedral groups.

\vskip 2mm

In Section $4$ we prove that the residual deficiency of $G$ minus $1$ is a supermultiplicative invariant with respect to finite index normal subgroups. That is
\begin{equation}\label{eq: supermult res def}
rdef(G)-1\leq \dfrac{rdef(H)-1}{|G:H|},
\end{equation}
where $H$ is a finite index normal subgroup of $G$. This allows us to define the residual deficiency gradient. We show, however, that this is equal to the deficiency gradient of the group.
\vskip 2mm

In Section $5$, we identify conditions that give a strict inequality in \cref{eq: supermult res def}. So, $G$ may have residual deficiency equal to one, but a strict inequality in \cref{eq: supermult res def} implies it has a finite index subgroup with residual deficiency greater than one. This is useful as most of the properties that interest us are invariant under finite index supergroups (for example being large or satisfying $\tilde{I}(G)>0$, where $\tilde{I}$ is any generalised Euler characteristic). These conditions, unfortunately, do not readily apply to any group $G$, but do to quotients $G/M$ where $G$ is not residually finite.

\vskip 2mm

\textbf{Acknowledgements}
\vskip 2mm

This work is part of the author's PhD done under the supervision of Jack Button. The author would like to thank Jack for all the support and advice needed for the completion of this work. Also, the author would like to thank the Mexican Council of Science and Technology (CONACyT) and the Cambridge Overseas Trust for their financial support all these years.

\section{Submultiplicative and supermultiplicative invariants\\ and generalised Euler characteristics}\label{section 1}

Consider a real valued invariant $I$ defined over a class of groups $\mathfrak{L}$. Let $\mathfrak{L}$ be closed under finite index subgroups. By this we mean that if $G$ is in $\mathfrak{L}$ and $H$ is a finite index subgroup of $G$, then $H$ is in $\mathfrak{L}$. The invariant $I$ is said to be \emph{submultiplicative} with respect to finite index (finite index normal) subgroups if 
\begin{equation}\label{eq: submult}
\dfrac{I(H)}{|G:H|}\leq I(G),
\end{equation}
where $G$ is in $\mathfrak{L}$, $H$ is a finite index (finite index normal) subgroup of $G$ and where $|G:H|$ denotes the index of $H$ in $G$.
\vskip 2mm

In an analogous way, $I$ is \emph{supermultiplicative} with respect to finite index (finite index normal) subgroups if
\[
\dfrac{I(H)}{|G:H|}\geq I(G).
\]
In problem $E10$ of \cite{wall}, C.T.C. Wall mentions that given a submultiplicative invariant $I$ as above, one may define
\begin{equation}\label{eq: gradient}
\tilde{I}(G)=\underset{H\underset{f}{\leqslant}G }{\text{inf}} \left\{ \dfrac{I(H)}{|G:H|} \right\}.
\end{equation}
He remarks, without giving proof, that if $H$ is a finite index subgroup of $G$, then 
\[
\tilde{I}(G)=\dfrac{\tilde{I}(H)}{|G:H|}. 
\]
An invariant such as $I$ which satisfies the previous equality with respect to finite index subgroups is called a \emph{generalised Euler characteristic}.
\vskip 2mm

In a similar vein, given a supermultiplicative invariant $I$, we define
\[
\tilde{I}(G)=\underset{H\underset{f}{\leqslant}G }{\text{sup}} \left\{ \dfrac{I(H)}{|G:H|} \right\}.
\]

\begin{note}
If $I$ is submultiplicative only with respect to finite index normal subgroups of $G$, then define
\[
\tilde{J}(G)=\underset{H\underset{f}{\trianglelefteq}G }{\text{inf}} \left\{ \dfrac{I(H)}{|G:H|} \right\}.
\]
In the case $I$ is supermultiplicative only with respect to finite index normal subgroups, then define
\[
\tilde{J}(G)=\underset{H\underset{f}{\trianglelefteq}G }{\text{sup}} \left\{ \dfrac{I(H)}{|G:H|} \right\}.
\]
\cref{prop: 2 section 2} below shows that it does not matter whether we consider finite index subgroups or only finite index normal subgroups.
\end{note}
\vskip 2mm

\begin{prop}\label{prop: 2 section 2}
Let $G$ be a group in a class of groups $\mathfrak{L}$ where $\mathfrak{L}$ is closed under finite index subgroups. Suppose $I$ is an invariant for groups in $\mathfrak{L}$, submultiplicative with respect to finite index subgroups. Then
\begin{equation}\label{eq: prop 2}
\tilde{I}(G)=\tilde{J}(G).
\end{equation}
If $I$ is submultiplicative with respect to finite index normal subgroups, then
\begin{equation}\label{eq; prop 2}
\tilde{I}(G)=\tilde{J}(G).
\end{equation}
\end{prop}
\begin{proof}
The set of finite index normal subgroups is contained in the set of finite index subgroups, therefore
\[
\underset{H\underset{f}{\leqslant}G }{\text{inf}} \left\{ \dfrac{I(H)}{|G:H|} \right\}
\leq\underset{H\underset{f}{\trianglelefteq}G }{\text{inf}} \left\{ \dfrac{I(H)}{|G:H|} \right\}.
\]
To prove the other inequality consider $H$ a finite index subgroup of $G$. Let $K$ be the core of $H$ in $G$, defined as the intersection of all the conjugates of $H$ in $G$. Note that $K$ is a finite index normal subgroup of $G$. Moreover, as $K\trianglelefteq_f H$,
then
\[
\dfrac{I(K)}{|H:K|}\leq I(H).
\]
This implies
\[
\dfrac{I(K)}{|G:K|}=\dfrac{I(K)}{|G:H||H:K|}\leq \dfrac{I(H)}{|G:H|}.
\]
As this holds for all finite index subgroup $H$ of $G$, the result follows.
\end{proof}
\vskip 2mm

\begin{note}
The previous proof equally works if $I$ is a supermultiplicative invariant. Just reverse the inequalities and change supremum for infimum in \cref{eq: prop 2} and \cref{eq; prop 2}.
\end{note}

\begin{prop}\label{prop: 1 section 2}
Let $G$ be a group in a class of groups $\mathfrak{L}$ where $\mathfrak{L}$ is closed under finite index subgroups. Suppose $I$ is an invariant defined over the groups in $\mathfrak{L}$ and assume $I$ is submultiplicative with respect to finite index normal subgroups of $G$. Then 
\begin{equation}\label{mult of I bar}
\tilde{I}(G)=\dfrac{\tilde{I}(H)}{|G:H|} 
\end{equation}
for all finite index subgroups $H$ of $G$. In particular $\tilde{I}$ is a generalised Euler characteristic.
\vskip 2mm

If $I$ is supermultiplicative with respect to finite index normal subgroups of $G$, then $\cref{mult of I bar}$ holds too.
\end{prop}
\begin{proof}
Assume $I$ is submultiplicative. The proof for when $I$ is supermultiplicative is the same but with the inequalities inverted.
\vskip 2mm

Consider $K$ a finite index normal subgroup of $G$. Let $N=K\cap H$ and consider $Core_G(N)$, the core of $N$ in $G$. Since $Core_G(N)$ is finite index in $H$, then \cref{eq: gradient} and \cref{eq: submult} imply
\[
\dfrac{\tilde{I}(H)}{|G:H|}\leq \dfrac{I\big(Core_G(N)\big)}{|G:H||H:Core_G(N)|}=\dfrac{I\big(Core_G(N)\big)}{|G:Core_G(N)|}
\]

\[
=\dfrac{I\big(Core_G(N)\big)}{|G:K||K:Core_G(N)|}\leq \dfrac{I(K)}{|G:K|}.
\]
\vskip 2mm

As this holds for all finite index normal subgroups $K$ of $G$, then
\[
\dfrac{\tilde{I}(H)}{|G:H|}\leq \tilde{I}(G).
\]
By \cref{prop: 2 section 2} assume $\tilde{I}$ is defined over all finite index subgroups. If $L$ is a finite index subgroup of $H$, then $L$ is a finite index subgroup of $G$. Hence, \cref{eq: gradient} implies
\[
\tilde{I}(G)\leq\dfrac{I(L)}{|G:L|}=\dfrac{I(L)}{|G:H||H:L|}.
\]
The previous inequality holds for all finite index subgroups of $H$, therefore
\[
\tilde{I}(G)\leq\dfrac{\tilde{I}(H)}{|G:H|}.
\]
\end{proof}

Next we give examples of invariants for finitely generated and finitely presented groups which are submultiplicative or supermultiplicative. The objective is to see how they relate to each other and set the stage for a new supermultiplicative invariant which will be introduced in the following section.
\vskip 2mm

The \emph{rank} $d(G)$ of a finitely generated group $G$ is defined as the minimum number of elements needed to generate $G$. Let $G$ be a finitely generated group generated by a finite set $X$ of size $n$
and let $H$ be a subgroup of finite index in $G$. The Nielson-Schreier method (\cite{lyndon}, pp. 16) computes a finite generating set $Y$ for $H$ and shows that
\[
\dfrac{|Y|-1}{|G:H|}=|X|-1.
\]
As the rank of $H$ is the minimal number of elements needed to generate the
group $H$ and the above equality holds for all generating sets $X$ of $G$,
then
\begin{equation}\label{ed;submultiplicativity of d}
\dfrac{d(H)-1}{|G:H|}\leq d(G)-1.
\end{equation}
\vskip 1mm

Therefore, the rank minus one is a submultiplicative invariant, with respect to finite index subgroups, for finitely generated groups.
\vskip 2mm

In \cite{lackers RG}, M. Lackenby defined the \emph{rank gradient} of a
finitely generated group $G$ as
\[
RG(G)=\underset{H\underset{f}{\leqslant}G }{\text{inf}} \left\{ \dfrac{d(H)-1}{|G:H|} \right\},
\]
where the infimum is taken over all the finite index subgroups of $G$.
The notion was motivated by problems in $3$-dimensional geometry and
has received lots of attention since then.
\vskip 2mm

C.T.C. Wall mentions in \cite{wall} that the rank of a group is an example of a submultiplicative invariant. At the time, however, little was known about the properties of the rank gradient nor had any computations been carried out. M. Lackenby was the first to study the rank gradient in more detail (\cite{lackers RG}), although it was mentioned in \cite{reznikov}.

\vskip 2mm

Having positive rank gradient (that is, $RG(G)>0$) is a strong property. By the Nielson-Schreier index formula $RG(F_n)=n-1$. In \cite{lackers RG}, M. Lackenby proved that if $G$ is a non-trivial free product $A\ast B$, where either $A$ or $B$ is not isomorphic to $C_2$, the cyclic group
of order two, then $RG(G)>0$. On the other hand, $SL_n(\mathbb{Z})$
for $n\geq 3$, ascending $HNN$-extensions, and direct products of
finitely generated infinite residually finite groups all have zero
rank gradient (\cite{lackers RG}, \cite{abert-nikolov}). In some sense, the rank gradient measures how free a group is.
\vskip 2mm

\cref{prop: 1 section 2} says that having positive rank gradient is invariant under finite index
subgroups and finite index supergroups and it shows that it defines a generalised Euler characteristic as mentioned in \cite{wall}.
\vskip 2mm

The rank gradient is related to the important notion of amenability. This notion has many different definitions. Here, we give the following one. A group $G$ is
\emph{amenable} if it has a finitely additive, left invariant,
probability measure. Being amenable is invariant under quotients,
extensions, subgroups and finite direct products. The notion was first introduced by J. von Neumann. Finite groups, finitely generated abelian groups and soluble groups are examples of amenable groups. On the other hand, any group which contains a non-abelian free group is non-amenable. Regarding the rank gradient, M. Lackenby proved in \cite{lackers RG}, that if a finitely presented residually finite group is amenable, then it has zero rank gradient. This result was strengthened by M. Abert and N. Nikolov in \cite{abert-nikolov} to finitely generated residually finite groups. 
\vskip 2mm

Another important property is Property ($\tau$). A finitely generated group has \emph{Property $(\tau)$} if for some (equivalently any) finite generating set $S$ for $G$, the set of Cayley graphs Cay($G/N,S$) form an expander family, where $N$ varies over all finite index normal subgroups of $G$. This property is preserved under quotients, extensions and subgroups of finite index. It is remarkable that this property has been used in many different areas such as graph theory, differential geometry, group theory and representation theory. Moreover, the theory relating to Property ($\tau$) is very rich (see for instance \cite{Lubotzky}).
\vskip 2mm

If $G$ is a finitely generated residually finite group which is amenable and has property ($\tau$), then $G$ is finite (\cite{Lubotzky}). In \cite{lackers RG}, M. Lackenby proved that if a group $G$ is finitely presented, then either it has positive rank gradient, it has Property ($\tau$), or an infinite number of finite index subgroups are HNN extensions or amalgamated free products.

\vskip 2mm

Say $G$ is a finitely presented group given
by a finite presentation $P=\langle X|R \rangle$. The deficiency of $P$ is defined as $|X|-|R|$. The deficiency of $G$ is defined as the supremum over all the possible finite presentations $P$ that $G$ admits.
\vskip 2mm

\begin{note}
Deficiency is well defined as the deficiency of a finite presentation
is bounded above by $\beta_1(G)$, where $\beta_1(G)$, the \emph{first Betti number of $G$}, is the number of copies of $\mathbb{Z}$ which appear in the
abelianisation of $G$. One way of seeing that the deficiency is bounded above by $\beta_1(G)$ is by computing the abelianisation of the group using Smith's normal form. 
\vskip 2mm

Consider a set $X$ of generators for $G$. Send this set under the canonical surjective homomorphism $f$ from $G$ onto its abelianisation. Then $f(X)$ generates the abelianisation of $G$. Therefore, $\beta_1(G)$ is a lower bound for $d(G)$. Hence
\begin{equation}\label{eq; def  beta rank}
def(G)\leq \beta_1(G)\leq d(G).
\end{equation}
\end{note}

\vskip 2mm

Consider $P$ a finite presentation for $G$. Let $H$ be a finite index subgroup of $G$. The Reidemeister-Schreier
rewriting method computes a
finite presentation $Q$ for $H$ and shows that
\[
def(P)-1= \dfrac{def(Q)-1}{|G:H|}.
\]
As the deficiency of $Q$ is less than or equal to the deficiency of $H$, then for all finite presentations $P$ of $G$
\[
def(P)-1= \dfrac{def(Q)-1}{|G:H|}\leq \dfrac{def(H)-1}{|G:H|}.
\]
As the deficiency of $G$ is the supremum over all finite
presentations, then
\begin{equation}\label{eq; def supermultiplicative}
def(G)-1\leq \dfrac{def(H)-1}{|G:H|}.
\end{equation}

Therefore, the deficiency minus one of a finitely presented group is a supermultiplicative invariant with respect to finite index subgroups.

\vskip 2mm

\begin{rem}\label{rem large fund gps not def>1}
As mentioned before, having deficiency greater than one is a strong property. As large groups contain a non-abelian free group, then they are non-amenable. Moreover, a large group does not have Property ($\tau$).
\vskip 2mm

The deficiency of a finitely presented group $G$ and its rank
gradient are nicely related. By \cref{eq; def  beta rank} and
\cref{eq; def supermultiplicative}
\[
def(G)-1\leq \dfrac{def(H)-1}{|G:H|}\leq \dfrac{d(H)-1}{|G:H|},
\]
where $H$ is a finite index subgroup of $G$. Therefore,
\[
def(G)-1\leq RG(G).
\]
Hence, if the finitely presented group $G$ has deficiency
greater than one, then the group has positive rank gradient.
\end{rem}

\vskip 2mm

Both largeness and having positive rank gradient are properties invariant under finite index supergroups. Therefore, in order to conclude both largeness and positive
rank gradient, it is enough to have a finite index subgroup
with deficiency greater than one. By \cref{eq; def supermultiplicative}, we define the \emph{Deficiency Gradient} as
\[
DG(G)=\underset{H\underset{f}{\leqslant}G }{\text{sup}}
\left\{ \dfrac{def(H)-1}{|G:H|} \right\}.
\]
\vskip 2mm

Note that $DG(G)>0$ if and only if $G$ has a finite index subgroup with deficiency greater than one. Therefore, having positive deficiency gradient implies the group $G$ is both large and has positive rank gradient.
\vskip 2mm

By \cref{prop: 1 section 2}, this invariant is multiplicative with respect to finite index subgroups and hence defines a generalised Euler characteristic.
\vskip 2mm

In \cite{puchta} (alternatively \cite{jack-anitha}), the notion of
$p$-deficiency was introduced. Let $G$ be a finitely generated
group given by a presentation $Q=\langle X|R\rangle$, where $X$ generates
$F_n$, the non-abelian group of rank $n$.  The author of \cite{puchta}
defined $\nu_{p}(r)$ for $r\in R$, to be the largest integer $k$,
such that there is some $s\in F_n$ with $r=s^{p^{k}}$. The $p$-deficiency
of $Q$ is defined in \cite{jack-anitha} as
\[
def_p(Q)=n-\sum_{r\in R}\ p^{-\nu_{p}(r)}.
\]
The definition of $p$-deficiency in \cite{puchta} is
\[
def_p(Q)=n-\sum_{r\in R}\ p^{-\nu_{p}(r)}-1.
\]
For example, let $Q$ be the following presentation
\[
\langle x_1,\ldots, x_n \mid w_1^{p^{a_1}},\ldots,w_m^{p^{m}}\rangle,
\]
where $w_i$ cannot be expressed as a proper $p$-power for all $i$, $1\leq i\leq m$. Then, the $p$-deficiency of $Q$ as defined in \cite{jack-anitha} is
\[
n-\sum_{i=1}^{m}\dfrac{1}{p^{a_i}},
\]
while the $p$-deficiency of $Q$ as defined in \cite{puchta} is
\[
n-\sum_{i=1}^{m}\dfrac{1}{p^{a_i}}-1.
\]
We opt for the first definition, that is, the one in \cite{jack-anitha}.
\vskip 2mm

We will be interested only in presentations $Q$ such that the quantity $\sum_{r\in R}\ p^{-\nu_{p}(r)}$ converges. That is, presentations for which $def_p(Q)$ is finite.
\vskip 2mm

Given $G$ a finitely generated group with a presentation $Q=\langle X|R\rangle$, where $X$ generates
$F_n$ and such that $def_p(Q)$ is finite, then the $p$-deficiency of $G$ is defined as the supremum of the $p$-deficiencies of all presentations of $G$ with finite generating
set.
\vskip 2mm

The authors of \cite{puchta-yiftach}, proved that $p$-deficiency minus one
is supermultiplicative with respect to finite index normal subgroups. Hence,
\[
def_p(G)-1\leq \dfrac{def_p(H)-1}{|G:H|},
\]
where $H$ is a finite index normal subgroup in $G$. Therefore, it makes sense to consider the supremum over all
normal finite index subgroups. Hence we define the
\emph{p-deficiency gradient} to be
\[
DG_p(G)=\underset{H\underset{f}{\trianglelefteq}G }{\text{sup}}
\left\{ \dfrac{def_p(H)-1}{|G:H|}  \right\}.
\]
By \cref{prop: 2 section 2} the quantity $DG_p$ would still be the same if the supremum is taken over finite index subgroups.
\vskip 2mm

Having $p$-deficiency greater than one is also a strong property. For instance, in \cite{jack-anitha} it is shown that a finitely presented group with $p$-deficiency greater than
one is $p$-large (a group is $p$-large if it has a normal 
subgroup of index a power of $p$ which surjects onto a non-abelian free
group).
\vskip 2mm

The deficiency gradient, $p$-deficiency gradient and rank gradient
are nicely related. Consider a finitely presented group $G$ and a finite presentation $P$ for $G$. Using the definitions of deficiency and $p$-deficiency for $P$, the deficiency of $P$ is less than or equal to its $p$-deficiency. Hence
\[
\dfrac{def(H)-1}{|G:H|}\leq \dfrac{def_p(H)-1}{|G:H|},
\]
where $H$ is a finite index normal subgroup of $G$. Therefore, $DG(G)\leq DG_p(G)$.
\vskip 2mm

Both deficiency and $p$-deficiency are supermultiplicative. Hence the inequality $DG(G)\leq DG_p(G)$ is rather straightforward. The rank gradient is however submultiplicative. Therefore, the following proposition comes as somewhat of a surprise.
\vskip 2mm

\begin{prop}\label{DGp less thn or eq RG}
Let $G$ be a finitely generated group. Then $DG_p(G)\leq RG(G)$.
\end{prop}
\begin{proof}
In \cite{puchta}, J.C. Schlage-Puchta proved that given a finitely
generated group $G$, then $def_p(G)\leq d_p(G)\leq d(G)$, where
$d_p(G)$, the $p$-rank of $G$, is defined as the maximum number of
copies of $C_p$, the cyclic group of order $p$, onto which $G$ can
surject.
\vskip 2mm

Let $H$ be a finite index normal subgroup of $G$. By the
supermultiplicativity of $p$-deficiency
\[
def_p(G)-1\leq \dfrac{def_p(H)-1}{|G:H|}\leq \dfrac{d_p(H)-1}{|G:H|}\leq \dfrac{d(H)-1}{|G:H|}.
\]
This holds for all normal finite index subgroups, therefore $def_p(G)-1\leq RG(G)$ and hence
\[
\dfrac{def_p(H)-1}{|G:H|}\leq \dfrac{RG(H)}{|G:H|}=RG(G),
\]
where the last equality follows from \cref{prop: 1 section 2}.
Take the supremum on the left hand side, over all finite
index normal subgroups of $G$, to obtain the result.
\end{proof}
\vskip 2mm

\begin{rem}\label{rem: p def >1 implies RG>0}
From \cref{DGp less thn or eq RG}, if $def_p(G)>1$, then $RG(G)>0$. This may also be deduced from results in \cite{puchta-yiftach}. Also note that having $DG_p(G)>0$ is equivalent to having a finite index normal subgroup with $p$-deficiency greater than one. Hence $DG_p(G)>0$ implies the group $G$ has a finite index normal subgroup $H$ which is $p$-large and has positive rank gradient. As $H$ has finite index in $G$, the latter is large and has positive rank gradient. 
\end{rem}
\vskip 2mm

\begin{cor}\label{cor: Def grad, pdef grad and rank grad}
Let $G$ be a finitely presented group. Then
\[
DG(G)\leq DG_p(G)\leq RG(G).
\]
\end{cor}

By Proposition $3.7$ in \cite{lackers covers of 3 orbs}, $d_p(G)-1$ defines a submultiplicative invariant for the class of finitely generated groups, with respect to subnormal subgroups of finite index a power of $p$. However, it does not define a submultiplicative or supermultiplicative invariant with respect to normal finite index subgroups, even if we restrict ourselves to the class of finitely presented groups.
\vskip 2mm

Consider a prime $p$ and the free product $C_{p_1}\ast C_{p_2}$, where $p_1$ and $p_2$ are primes distinct from $p$ and at least one of $p_1$ and $p_2$ is greater than $2$. The abelianisation of $C_{p_1}\ast C_{p_2}$ is $C_{p_1}\times C_{p_2}$, therefore its $p$-rank is zero. However, $C_{p_1}\ast C_{p_2}$ has a non-abelian free group $F$ of finite rank as a finite index subgroup.
\vskip 2mm

Any non-abelian free group has positive $p$-rank for any prime $p$. Therefore,
\[
0<\dfrac{d_p(F)-1}{|C_{p_1}\ast C_{p_2}:F|}.
\]
However, $d_p(C_{p_1}\ast C_{p_2})-1=-1$. Therefore, the $p$-rank does not define a submultiplicative invariant with respect to finite index subgroups.
\vskip 2mm

Let $p$ and $G=C_{p_1}\ast C_{p_2}$ be as above . Consider $\Gamma=G\times C_p$. The $p$-rank of $G$ is zero. Note $G$ is finite index in $\Gamma$. However, the $p$-rank of $\Gamma$ is at least one. Therefore $d_p(\Gamma)-1\geq 0$, but
\[
\dfrac{d_p(G)-1}{|\Gamma:G|}<0.
\]
Therefore, the $p$-rank does not define a supermultiplicative invariant with respect to finite index subgroups.
\vskip 2mm

Even when the $p$-rank does not define a supermultiplicative or a submultiplicative invariant, with respect to finite index subgroups, the way in which it varies on its finite index subgroups has been useful in some cases. In \cite{lackers hom growth of subgps}, M. Lackenby considered collections of finite index subgroups $\{G_i\}$ for a finitely presented group $G$ and said that $\{G_i\}$ has linear growth of mod $p$ homology if
\begin{equation}\label{eq: lackers def with p rank}
\underset{i}{\text{inf}}\left\{\dfrac{d_p(G_i)}{|G:G_i|}\right\}>0.
\end{equation}

M. Lackenby used this condition to give a characterisation for largeness in finitely presented groups (\cite{lackers hom growth of subgps}).
\vskip 2mm

In \cite{puchta-yiftach}, Y. Barnea and J.C. Schlage-Puchta considered the quantities
\[
\alpha^{-}(G)=\text{lim\ inf}\dfrac{d_p(H)}{|G:H|}
\]
and
\[
\alpha^{+}(G)=\text{lim\ sup}\dfrac{d_p(H)}{|G:H|},
\]
where the lim inf and the lim sup are taken over all finite index normal subgroups of $G$.
\vskip 2mm

The authors of \cite{puchta-yiftach} proved that if a finitely generated group $G$ has any of the two quantities above being positive, say $\alpha^{-}(G)$, then $G$ has a normal subgroup $N$ such that $G/N$ is a $p$-group and $\alpha^{-}(G/N)>0$.
\vskip 2mm

We define the $p$-rank gradient as
\[
RG_p(G):= \underset{H\underset{f}{\trianglelefteq}G }{\text{inf}} \left\{ \dfrac{d_p(H)-1}{|G:H|} \right\}.
\]
\begin{prop}\label{lemma pdef>1 implies RG_p>0}
Let $G$ be a finitely generated group. Then $def_p(G)-1\leq RG_p(G)$. In particular, if $def_p(G)>1$, then $RG_p(G)>0$.
\end{prop}
\begin{proof}
In \cite{puchta-yiftach} it was proved that $p$-deficiency minus one is supermultiplicative with respect to finite index normal subgroups of $G$. That is, if $H$ is a finite index normal subgroup of $G$, then
\[
def_p(G)-1\leq \dfrac{def_p(H)-1}{|G:H|}.
\]
Moreover, the $p$-deficiency of a group is a lower bound for the $p$-rank of a group (\cite{puchta}), hence
\begin{equation}\label{eq: in first lemma}
def_p(G)-1\leq \dfrac{def_p(H)-1}{|G:H|}\leq \dfrac{d_p(H)-1}{|G:H|}.
\end{equation}
The result is obtained by taking the infimum over all the finite index normal subgroups of $G$.
\end{proof}
\vskip 2mm

As the $p$-rank gradient and the rank gradient are defined by using the infimum, and $d_p(G)\leq d(G)$, then $RG_p(G)\leq RG(G)$, where $G$ is a finitely generated group. From \cref{lemma pdef>1 implies RG_p>0} we then have
\[
def_p(G)-1\leq RG_p(G)\leq RG(G).
\]

\vskip 2mm

Now we introduce another notion of deficiency which we call group deficiency. This notion is partly inspired by the following result in \cite{thomas} by R. Thomas.
\vskip 2mm

\begin{thm}\label{thm thomas}$($Only Theorem in $\cite{thomas})$\\
Let $G$ be given by the presentation
\[
Q=\langle x_1,\ldots,x_n \mid w_1^{m_1},\ldots,w^{m_s}_s\rangle,
\]
and let $\varphi: F_n\longrightarrow G$ be the canonical map from the non-abelian
free group of rank $n$ onto $G$. If the order of $\varphi(w_i)$  in $G$ is $m_i$,
for all $i$, $1\leq i\leq s$, and
\[
n-\sum_{i=1}^{s}\dfrac{1}{m_i}\geq 1,
\]
then the group $G$ is infinite.
\end{thm}
This theorem considers the quantity $n-\sum_{i=1}^{s}1/m_i$, which is different to the quantities considered in the deficiency and the $p$-deficiency. This leads to the following definition.
\vskip 2mm

\begin{defn}
Let $G$ be a finitely presented group given by a presentation $Q=\langle x_1,\ldots,x_n \mid w_1^{m_1},\ldots,w^{m_s}_s\rangle$. Let $\varphi: F_n\longrightarrow G$ be the canonical map from the non-abelian free group of rank $n$ onto $G$, induced by the presentation $Q$. Suppose the order of $\varphi(w_i)$ in $G$ is $k_i$, for all $i$, $1\leq i\leq s$. Then we define the \emph{group deficiency} of $Q$ to be
\[
n-\sum_{i=1}^{s}1/k_i.
\]
We define the group deficiency of $G$ as the supremum of all the group deficiencies taken over all possible finite presentations of $G$.
\end{defn}
\vskip 2mm

Note that the group deficiency compares to the deficiency in the following way
\[
def(Q)\leq gdef(Q).
\]
\vskip 2mm

It is not clear whether the group deficiency defines a submultiplicative or supermultiplicative invariant. We can nevertheless define the \emph{Group Deficiency Gradient} as
\[
GDG(G)=\underset{H\underset{f}{\leqslant}G }{\text{sup}}
\left\{ \dfrac{gdef(H)-1}{|G:H|} \right\}.
\]
The previous definition implies that 
\begin{equation}\label{eq: def grad and gp def grad}
DG(G)\leq GDG(G).
\end{equation}
\vskip 2mm

Let $G$ be given by the presentation
\[
Q=\langle x_1,\ldots,x_n \mid u_i^{m_i} \rangle,\ \ i\in I,
\]
where $I$ is a countable set. Suppose $k_i$ is the order of the image of $u_i$ in $G$, for all $i\in I$. If $\sum_{i\in I}1/k_i$ converges, then it makes sense to define the group deficiency of $Q$ as $n-\sum_{i\in I}1/k_i$. This helps extend the notion beyond the class of finitely presented groups.
\vskip 2mm

Another invariant which is useful to keep track of on finite index
subgroups is the first Betti number. In \cite{eckman}, the
first $L^2$-Betti number $\beta^{(2)}_1(G)$ is defined for a finitely generated group $G$. For definition and main properties we refer to \cite{eckman}. For the moment, it suffices to note that the first $L^2$-Betti number satisfies
\begin{equation}\label{eq; supermult betti}
\beta^{(2)}_1(G)=\dfrac{\beta^{(2)}_1(H)}{|G:H|}
\end{equation}
for all finite index subgroups $H$ in $G$. Therefore, it defines a generalised Euler characteristic. 
\vskip 2mm

When restricted to finitely presented residually finite groups, L\"{u}ck's approximation Theorem (\cite{luck approx thm}) helps us reduce the definition of the first $L^2$-Betti number to the following: given a finitely presented residually finite group $G$ and a descending sequence of
finite index normal subgroups $\{N_i\}_{i\in\mathbb{N}}$, such that
$\underset{i\in\mathbb{N}}{\bigcap}N_i=\{e\}$, then
\[
\beta^{(2)}_1(G)=\underset{i\rightarrow\infty}{lim}\dfrac{\beta_1(N_i)}{|G:N_i|}.
\]
\vskip 2mm

In \cite{luck}, the finitely generated case was addressed.
It turns out from \cite{luck} that if the group $G$ is finitely generated and residually finite,
then
\begin{equation}\label{eq;approx L2 betti for fg gps}
\beta^{(2)}_1(G)\geq\underset{i\rightarrow\infty}{lim}sup\dfrac{\beta_1(N_i)}{|G:N_i|}.
\end{equation}
Moreover, the inequality cannot be improved as examples
where the strict inequality holds are given in \cite{luck}.
\vskip 2mm

The authors of \cite{kar niblo} used the group deficiency of a finitely generated group to give the following result
\begin{thm}$(\cite{kar niblo}$, Theorem 4$)$\\
Suppose a group G is given by the presentation
\[
Q=\langle x_1,\ldots,x_n \mid u_i^{m_i} \rangle,\ \ i\in I,
\]
where $I$ is a countable set and the order of $u_i$ in $G$ is $m_i$ for all $i\in I$. Suppose $\sum_{i\in I}1/m_i$ converges. If $G$ is infinite, then 
\[
\beta_1^{2}(G)\geq n-1-\sum_{i\in I}\dfrac{1}{m_i}.
\]
\end{thm}

\vskip 2mm
\begin{cor}\label{cor: GDG and l2 betti}
Let $G$ be a finitely presented group $G$. Then
\[
GDG(G)\leq\beta_1^{2}(G).
\]
\end{cor}
\begin{proof}
Let $H$ be a finite index subgroup of $G$. As $G$ is finitely presented, then $H$ is finitely presented. Therefore, by the previous theorem we have
\[
\beta_1^{2}(H)\geq gdef(H)-1.
\]
Since the first $L_2$-Betti number is multiplicative with respect to finite index subgroups
\[
\beta_1^{2}(G)=\dfrac{\beta_1^{2}(H)}{|G:H|}\geq \dfrac{gdef(H)-1}{|G:H|}.
\]
The result is obtained by taking the supremum, on the right hand side, over all finite index subgroups of $G$.

\end{proof}
\vskip 2mm

\begin{cor}\label{cor: DG GDG and L2 betti}
Let $G$ be a finitely presented group. Then
\[
DG(G)\leq GDG(G)\leq\beta_1^{2}(G).
\]
\end{cor}
\begin{proof}
Immediate from \cref{eq: def grad and gp def grad} and \cref{cor: GDG and l2 betti}.
\end{proof}

As mentioned before, the rank gradient, the deficiency gradient and the first $L^{2}$-Betti number define generalised Euler characteristics. The authors of \cite{puchta-yiftach} considered
\[
\chi_p(G)=-\underset{|G:H|<\infty}{sup}\dfrac{def_p(H)-1}{|G:H|}
\]
and proved the following result.

\begin{prop}$(\cite{puchta-yiftach}$, Proposition 17$)$.\\
Let $G$ be a finitely generated group.
\begin{enumerate}
\item Then 
\[
\chi_p(G)=-\underset{\underset{H\triangleleft G}{|G:H|<\infty}}{sup}\dfrac{def_p(H)-1}{|G:H|}.
\]
\item If $H$ is a finite index subgroup of $G$, then $\chi_p(H)=|G:H|\chi_p(G)$.
\item If $G$ is virtually free, then $\chi_p(G)$ is the ordinary
Euler characteristic of $G$.
\item If $G$ is a Fuchsian group, then $\chi_p(G)$ is the ordinary Euler characteristic of $G$.
\end{enumerate}
\end{prop}
\vskip 2mm

The following propositions show that the deficiency gradient, rank gradient and first $L^{2}$-Betti number, satisfy the same properties as $\chi_p(G)$. Note that the first two properties have already been settled in \cref{prop: 2 section 2} and \cref{prop: 1 section 2}, hence only the last two are proved.
\vskip 2mm

\begin{prop}\label{Euler char and DG}
Let $G$ be a finitely presented group. Then if $G$ is virtually free or a Fuchsian group, then $-DG(G)$ is the ordinary Euler characteristic of $G$.
\end{prop}
\begin{proof}
Denote the ordinary Euler characteristic of $G$ by $\chi(G)$. The Euler characteristic and the deficiency gradient are multiplicative. Therefore, when $G$ is virtually free, it suffices to prove the statement for non-abelian free groups. Also, as every Fuchsian group has a surface group as a finite index subgroup, it suffices to prove the statement for surface groups when $G$ is a Fuchsian group.
\vskip 2mm

Suppose $G$ is free of rank $m$. Then the deficiency of $G$ is $m$. Let $H$ be a finite index subgroup of $G$. By the Nielson-Schreier index formula $H$ is a free group of rank $|G:H|(m-1)+1$ and hence has deficiency $|G:H|(m-1)+1$. Therefore,
\[
-DG(G)=-\underset{H\underset{f}{\leqslant}G}{\text{sup}} \left\{ \dfrac{def(H)-1}{|G:H|} \right\}=-\underset{H\underset{f}{\leqslant}G}{\text{sup}} \left\{ \dfrac{1+|G:H|(m-1)-1}{|G:H|} \right\}
\]
\[
=1-m=\chi(G).
\]

Let $G$ be $\pi_1(S_g)$, where $S_g$ is the orientable surface of genus $g>1$. The deficiency of $\pi_1(S_g)$ is $2g-1$. If $H$ is a finite index subgroup of $G$, then $H\cong \pi_1(S_h)$, where $h=|G:H|(g-1)+1$. For $l>1$, the Euler characteristic of $S_l$ is $(2-2l)$. Therefore,
\[
-DG(G)=-\underset{H\underset{f}{\leqslant}G}{\text{sup}} \left\{ \dfrac{def(H)-1}{|G:H|} \right\}
\]
\[
=-\underset{H\underset{f}{\leqslant}G}{\text{sup}} \left\{ \dfrac{\Big(2\big(|G:H|(g-1)+1\big)-1\Big)-1}{|G:H|} \right\}=2-2g.
\]
\end{proof}
\vskip 2mm

\begin{prop}
Let $G$ be a finitely generated group. If $G$ is virtually free or a Fuchsian group, then $-RG(G)$ is the ordinary
Euler characteristic of $G$.
\end{prop}
\begin{proof}
If $G$ is virtually free, then the proof in \cref{Euler char and DG} works equally well. Just note that if $G$ is free of rank $m$ then the deficiency of $G$ is $m$.
\vskip 2mm

Let $G$ be a Fuchsian group. Suppose, as in \cref{Euler char and DG}, that $G\cong \pi_1(S_g)$, $g>1$. The rank of $\pi_1(S_g)$ is $2g$. If $H$ is a finite index subgroup of $G$, then $\pi_1(S_h)\cong H$, where $h=|G:H|(g-1)+1$. Therefore, the rank of $H$ is $2h$. Hence,

\[
-RG(G)=-\underset{H\underset{f}{\leqslant}G}{\text{inf}} \left\{ \dfrac{d(H)-1}{|G:H|} \right\}=-\underset{H\underset{f}{\leqslant}G}{\text{inf}} \left\{ \dfrac{2\big(|G:H|(g-1)+1\big)-1}{|G:H|} \right\}
\]
\[
=-\underset{H\underset{f}{\leqslant}G}{\text{inf}} \left\{ \dfrac{2+2|G:H|(g-1)-1}{|G:H|} \right\}=2-2g.
\]
\end{proof}
\vskip 2mm

\begin{prop}
Let $G$ be a finitely generated group. If $G$ is virtually free or a Fuchsian group, then $-\beta_1^{(2)}(G)$ is the ordinary Euler characteristic of $G$.
\end{prop}
\begin{proof}
Once again, it suffices to show the statement for non-abelian free groups and surface groups. Let $F_k$ be the non-abelian free group of rank $k$. Consider a descending sequence of finite index normal subgroups in $F_k$ that intersects in the identity. This can be done as $F_k$ is residually finite. Note that $N_i$ is a non-abelian free group, therefore, $\beta_1(N_i)=d(N_i)=|F_k:N_i|(k-1)+1$. Hence,
\[
-\beta^{(2)}_1(F_k)=-\underset{i\rightarrow\infty}{lim}\dfrac{\beta_1(N_i)}{|F_k:N_i|}=-\underset{i\rightarrow\infty}{lim}\dfrac{|F_k:N_i|(k-1)+1}{|F_k:N_i|}
=1-k=\chi(F_k).
\]

Suppose $G\cong \pi_1(S_g)$ with $g>1$. Consider a descending sequence $\{G_i\}_{i\in \mathbb{N}}$ of finite index normal subgroups in $G$ that intersects in the identity. Such series exist as $G$ is residually finite. Note that since $G$ is a surface group, then $G_i$ is a surface group for all $i\in \mathbb{N}$. 
\vskip 2mm

The Euler characteristic is multiplicative with respect to finite index subgroups. Therefore, $2g_i-2=|G:G_i|(2g-2)$ where $g_i=|G:G_{i}|(g-1)+1$. Finally, note that $\beta_1(G_i)=2g_i$. Therefore,
\[
-\underset{i\rightarrow\infty}{lim}\dfrac{\beta_1(G_i)}{|G:G_i|}=-\underset{i\rightarrow\infty}{lim}\dfrac{2g_i}{|G:G_i|}=-\underset{i\rightarrow\infty}{lim}\dfrac{2|G:G_i|(g-1)+2}{|G:G_i|}=2-2g.
\]
\end{proof}

\section{Residual deficiency}

This section introduces an invariant for finitely presented
groups that we call \emph{residual deficiency}. From now on, unless otherwise stated, any group considered is finitely presented. If $G$ is a finitely presented group and $Q=\langle X|R\rangle$ is a finite presentation for $G$, where $X$ freely generates $F_n$, the
non-abelian free group of rank $n$, denote by $\varphi$ the canonical homomorphism from $F_n$ to $G$. Denote by $R_G$ the finite residual of $G$, that is, the intersection of all finite index subgroups of $G$, and call
$G/R_G$ the \emph{residual quotient} of $G$.

\vskip 2mm

\begin{defn}\label{def: residual def}
Let $G$ be a finitely presented group and let $Q=\langle X|R\rangle$ be a
finite presentation for $G$. Consider the elements of
$R$ expressed in the following way $R=\{r_1=u_1^{m_1},\ldots,r_s=u_s^{m_s}\}$, where the elements
$u_i\in F_n$ cannot be expressed as proper powers of other words
in $F_n$. Let $\overline{R}_G$ be the inverse image of $R_G$ under $\varphi$ and let $k_i$ be the order of $u_i$ in $F_n/\overline{R}_G$($\cong G/R_G$). 
We define the \emph{residual deficiency
of the finite presentation $Q$} to be
\[
rdef(Q)=n-\sum_{i=1}^{s}\dfrac{1}{k_i}.
\]
Note that $k_i$ is also equal to the order of $\varphi(u_i)$ in $G/R_G$.
\vskip 2mm

Consider all possible finite presentations for $G$ and their corresponding residual deficiencies. We define the \emph{residual deficiency} of $G$ as the supremum of these residual deficiencies,
\[
rdef(G)=\underset{\langle X|R\rangle\cong G}{\text{sup}} \left\{ rdef(Q) \right\}.
\]
\end{defn}
\vskip 2mm

\begin{rem}\label{rem: rdef and def}
The first thing to observe is that if $Q$ is a finite presentation for a group $G$, then $def(Q)\leq rdef(Q)$. Hence, $def(G)\leq rdef(G)$. However, if the residual quotient of $G$ is torsion free, then $def(G)=rdef(G)$. This is because the only way in which the residual deficiency may be greater than the deficiency is if there is a presentation in which one of the relators is expressed as a proper power of another element which lies outside of $\overline{R}_G$. This element then defines a non-trivial torsion element in the residual quotient. 
\end{rem}
\vskip 2mm

The following lemma (or variations of it) is a well known result that appears in many papers, for instance Lemma $2.3$ of \cite{olshanskii-osin}.
\vskip 2mm

\begin{lem}\label{lemma}
Let $G$ be a finitely generated group, $H$ a normal subgroup of
finite index in $G$, and $g$ an element of $H$. Denote the centraliser of $g$ in $G$ by $C_G(g)$. Consider $K$ a subgroup of $C_G(g)$. Let $T$ be a left transversal for $HK$ in $G$ and let $Z$ be the set of conjugates of $g$ by elements in
$T$. Then, the normal subgroup in $G$ generated by $g$, is the
normal subgroup in $H$ generated by $Z$.
\end{lem}

\begin{proof}
Suppose $HK/H$ has order $n$. Note that $H\trianglelefteq HK\leqslant G$. Therefore, if $H$ has index $j$ in $G$, then $HK$ has index $j/n$ in $G$. 

\vskip 2mm

The proof centres around showing that if $T=\{t_{1},\ldots,t_{j/n} \}$ is a left transversal for $HK$
in $G$, then $\langle\langle
g\rangle\rangle^{G}=\langle\langle t_{l} g
t_{l}^{-1}\rangle\rangle^{H}$, where $1\leq l\leq j/n$. For this it suffices to show that $s g s^{-1}$ is in $\langle\langle t_l g
t_l^{-1}\rangle\rangle^{H}$, for all $s$ in $G$.

\vskip 2mm

Note that $s$ can be written as $t_{l}hk$ for some $t_{l}$ in $T$,
$h$ in $H$ and $k$ in $K$, where $K$ is in the centraliser of $g$ in $G$. Hence,
\[
s g s^{-1}=t_lhk g k^{-1}h^{-1}t_l^{-1}=t_l h g h^{-1}t_{l}^{-1}
=h't_l g t_{l}^{-1}h'^{-1}
\]
where $h'=t_{l}ht_{l}^{-1}$ as $H$ is normal in $G$. Thus $s g
s^{-1}\in \langle\langle t_{l} g t_{l}^{-1}\rangle\rangle^H$.
\end{proof}
\vskip 2mm

The following Theorem is central in this paper. Similar forms of this result may be found in the literature. For instance in \cite{Edjvet-balanced} and in \cite{gen traingle groups}.
\vskip 2mm

\begin{thm}\label{deficiency for finite index subgroups}
Let the group $G$ be given by the presentation
\[
P=\langle x_1,\ldots,x_n \mid u_1^{m_1},\ldots,u_r^{m_r}\rangle,
\]
where $n, m_i\geq 1$ for $i=1,\ldots,r$. Let $k_i$, for $1\leq i\leq r$, be as in \cref{def: residual def}, the order of $\varphi(u_i)$ in $G/R_G$. Then, there are finite index normal subgroups $H$ in $G$, such that the order of $\varphi(u_i)$ in $G/H$ is $k_i$, for all $i$, $1\leq i\leq r$. Moreover, the deficiency of every such $H$ is bounded below by
\[
1+|G:H|(rdef(P)-1).
\]
\end{thm}
\begin{proof}As $P$ is a presentation for $G$, there is a surjective
homomorphism $\varphi$ from $F_n$, the non-abelian free group of rank
$n$, onto $G$, where the kernel $N$ is the normal subgroup of
$F_n$ generated by $u_1^{m_1},\ldots,u_r^{m_r}$.

\vskip 2mm

Denote by $\overline{R}_G$ the inverse image of $R_G$ under
$\varphi$. As $F_n/\overline{R}_G\cong G/R_G$, then $F_n/\overline{R}_G$ is residually finite. Therefore, for each word $u_i^{m_i}$, there is
a finite index normal subgroup $\overline{H_i}$ in $F_n$, such that
$o(u_i,F_n/\overline{R}_G)=o(u_i,F_n/\overline{H_i})$. Let $\overline{H}$ be the intersection of all the $\overline{H_i}$. Then, as there are only finite number of relators in $P$, the subgroup $\overline{H}$ has finite index in $F_n/\overline{R}_G$ and $o(u_i,F_n/\overline{R}_G)=o(u_i,F_n/\overline{H})$ for all $i$, $1\leq i\leq r$. Note that every finite index normal subgroup in $G$ contained in $H$ also satisfies this condition.

\vskip 2mm

Denote by $H$ the normal subgroup in $G$ that corresponds to $\overline{H}$ under $\varphi$. Note that $|G:H|=|F_n:\overline{H}|$. Compute a presentation for $H$ using the Reidemeister-Schreier rewriting process (pp. 103, \cite{lyndon}). This gives a presentation with $|G:H|(n-1)+1$ generators which freely generate $\overline{H}$. To compute the relators note that the subgroup $N$ is generated as a normal subgroup of $\overline{H}$ by elements of the type $su_i^{m_i}s^{-1}$, where $s\in T$ and $T$ is a transversal of $\overline{H}$ in $F_n$. Since $\overline{H}\trianglelefteq \overline{H}\langle u_i\rangle\leqslant F_n$, then every $s$ in $T$ can be written as $thu_i^l$, with $t\in T_i$, where $T_i$ a left transversal of $\overline{H}\langle u_i\rangle$ in $F_n$, $h\in \overline{H}$, and $0\leq l< o(u_i,F_n/\overline{H})$. Moreover, the subgroup in $F_n$ generated by $u_i$ is a subgroup of $C_{F_n}(u_i)$. Therefore, by \cref{lemma}, each element $u_i^{m_i}$ only needs to be conjugated by $|G:H|/o(u_i,F_n/\overline{H})$ elements. Hence, $H$ admits a presentation with deficiency
\[
|G:H|(n-1)+1-\sum_{i=1}^{r}\dfrac{|G:H|}{o(u_i,F_n/\overline{H})}=1+|G:H|(rdef(P)-1).
\]
\end{proof}

\begin{cor}\label{cor positive res def}
Let $G$ be a finitely presented group. Then $rdef(G)-1\leq DG(G)$.
In particular, if $rdef(G)>1$, then $G$ has a finite index subgroup with deficiency greater than one.
\end{cor}
\begin{proof}
By \cref{deficiency for finite index subgroups}
\[
rdef(P)-1\leq \dfrac{def(H)-1}{|G:H|}\leq DG(G).
\]
Taking the supremum over all finite presentations $P$ of $G$ yields
\[
rdef(G)-1\leq DG(G).
\]
If $rdef(G)>1$, then $DG(G)>0$, which is equivalent to having a finite index subgroup with deficiency greater than one.
\end{proof}

\begin{rem}
From \cref{eq: def grad and gp def grad} $DG(G)\leq GDG(G)$, therefore,
\[
rdef(G)-1\leq DG(G)\leq GDG(G).
\]
\end{rem}
\vskip 2mm

\begin{ex}\label{ex: resdef-1 < DG}
This example shows that the equality $rdef(G)-1\leq DG(G)$ is not a strict equality in general. Consider the Thompson group $T$ given by the following presentation
\begin{equation}\label{eq: thomson pres}
\langle a,b\mid [ab^{-1},a^{-1}ba],[ab^{-1},a^{-2}ba^2]\rangle.
\end{equation}
One of the main properties of this group is that its commutator subgroup is equal to its residual subgroup. Such groups are known as residually abelian. For this and other properties of the Thompson group we refer to \cite{thomson gp notes}.
\vskip 2mm

As the relations in \cref{eq: thomson pres} are in the commutator, the abelianisation of $T$ is isomorphic to $\mathbb{Z}\times\mathbb{Z}$. Since the commutator subgroup of $T$ is the same as its residual subgroup, then the commutator $[a,b]$ generates $\overline{R}_T$ as a normal subgroup of $F_2$.
\vskip 2mm

Consider $G:=T\ast \mathbb{Z}$. Denote the normal closure of $[a,b]$ in $F_3=\langle a,b,c \rangle$ by $N$. Since $\overline{R}_G$ is normal in $F_3$ and $[a,b]\in\overline{R}_T\trianglelefteq\overline{R}_G$, then $\overline{R}_T\trianglelefteq N\trianglelefteq \overline{R}_G$. A presentation for $G/N$ is given by
\[
\langle a,b,c\mid [a,b]\rangle.
\]
This defines a group isomorphic to $(\mathbb{Z}\times\mathbb{Z})\ast\mathbb{Z}$. Since $N\trianglelefteq\overline{R}_G$ and $G/N\cong (\mathbb{Z}\times\mathbb{Z})\ast\mathbb{Z}$ is residually finite, then $N=\overline{R}_G$. As $G/R_G$ is isomorphic to $(\mathbb{Z}\times\mathbb{Z})\ast\mathbb{Z}$, which is torsion free, then $def(G)=rdef(G)$ by \cref{rem: rdef and def}. In \cite{button lerf}, J. Button showed that $def(G)=1$. Therefore, $rdef(G)=1$.

\vskip 2mm

Finally, the Kurosch subgroup theorem allows us to conclude that $T\ast\mathbb{Z}\ast\mathbb{Z}$ is a subgroup of finite index in $G$. Since $T\ast\mathbb{Z}\ast\mathbb{Z}$ has deficiency greater than one, then $G$ has a finite index subgroup with deficiency greater than one and hence $DG(G)>0$. Therefore, $G$ satisfies
\[
rdef(G)-1<DG(G).
\]
\end{ex}
\vskip 2mm

\cref{deficiency for finite index subgroups} implies the following
result by D. Allcock (\cite{allcock}).
\vskip 2mm

\begin{cor}\label{cor: allcock}$(\cite{allcock}$, only theorem$)$.\\
Let the group $G$ be given by the presentation
\[
P=\langle x_1,\ldots,x_n \mid u_1^{m_1},\ldots,u_r^{m_r}\rangle.
\]
Suppose $H$ is a normal subgroup of $G$ of index $N<\infty$, and that for each $j$, $1\leq j\leq r$, we have $u^{k}_j\notin H$ for $k=1,\ldots,m_j-1$. Then the rank of the abelianisation of $H$ is at least
\[
1+|G:H|(n-1-\sum_{j=1}^{r}\dfrac{1}{m_j}).
\]
\end{cor}
\begin{proof}
By \cref{deficiency for finite index subgroups}, the deficiency of $H$ is at least
\[
1+|G:H|(rdef(P)-1),
\]
where the latter is greater than or equal to
\[
1+|G:H|(n-1-\sum_{j=1}^{r}\dfrac{1}{m_j}).
\]
The result follows as the deficiency is a lower bound for the rank of the abelianisation.
\end{proof}
\vskip 2mm

\begin{rem}
In \cite{allcock}, \cref{cor: allcock} is used to conclude \cref{thm thomas} proved by R. Thomas in \cite{thomas}. As \cref{deficiency for finite index subgroups} implies \cref{cor: allcock}, then \cref{deficiency for finite index subgroups} also implies \cref{thm thomas}.
\end{rem}
\vskip 2mm

\begin{cor}\label{cor: rdef DG PDG and RG ineq}
Let $G$ be a finitely presented group. Then
\[
rdef(G)-1\leq DG(G)\leq DG_p(G)\leq RG(G),
\]
\[
rdef(G)-1\leq DG(G)\leq GDG(G)\leq \beta^{2}_1(G).
\]
In particular, if the residual deficiency of $G$ is greater than one, then $G$ has positive rank gradient, positive first $L^2$-Betti number and is large.

\end{cor}
\begin{proof}
The inequalities follow from Corollary \ref{cor: Def grad, pdef grad and rank grad}, \cref{cor: DG GDG and L2 betti} and 
\cref{cor positive res def}. Positive rank gradient and first $L^2$-Betti number are immediate from the inequalities. For largeness, if $rdef(G)>1$, then $DG(G)>0$, which is equivalent to having a finite index subgroup with deficiency greater than one. The latter implies largeness by \cite{baums-pride}.
\end{proof}
\vskip 2mm

The residual deficiency of a finite presentation can easily be computed or at least approximated in some cases. To illustrate this, we first need the following definition and lemma.
\vskip 2mm

\begin{defn}\label{def: no collapse}
We say that the finite presentation
\[
P=\langle x_1,\ldots,x_n \mid r_1^{a_1},\ldots, r_m^{a_m} \rangle
\]
has a \emph{quotient with no collapse} if $G$, the group defined by $P$, has a homomorphism $\psi$ from $G$ to a group $N$, such that the order of $\psi\circ\varphi(r_i)$ in $N$ is $a_i$, for all $i$, $1\leq i\leq m$, where $\varphi$ is the canonical map induced by $P$ from $F_n$ to $G$.
\end{defn}
\vskip 2mm

Consider the particular case when $N$ is residually finite. In this case we say $G$ has a \emph{residually finite quotient with no collapse}. 
\vskip 2mm

\begin{lem}\label{lem: no collapse}
Let $G$ be given by a presentation $P=\langle x_1,\ldots,x_n \mid r_1^{a_1},\ldots, r_m^{a_m} \rangle$, where the words $r_i$ are not expressed as proper powers, for $1\leq i\leq m$. Suppose $G$ has a residually finite quotient with no collapse $\psi$. Then
\[
rdef(P)=n-\sum_{i=1}^{m}\dfrac{1}{a_i}.
\]
\end{lem}
\begin{proof}
Let $K$ be the kernel of $\psi$. The condition that the order of $\psi\circ\varphi(r_i)$ in $N$ is $a_i$ says that the order of the image of $\varphi(r_i)$ in $G/K$ is $a_i$. As the image of $\psi$ is residually finite, the quotient $G/K$ is residually finite and hence the finite residual $R_G$ is contained in $K$. Therefore, the order of the image of $\varphi(r_i)$ in $G/R_G$ is $a_i$ for all $i$, $1\leq i\leq m$. Since the words $r_i$ cannot be expressed as proper powers, then the residual deficiency of $P$ is $n-\sum_{i=1}^{m}1/a_i$.
\end{proof}
\vskip 2mm

\begin{ex}\label{ex: gen triangle gps}
Consider generalised triangle groups. These are groups given by presentations of the form
\[
\langle a,b \mid a^l,b^m,w^n\rangle,
\]
where
\[
w = a^{r_1} b^{s_1}\cdots a^{r_k} b^{s_k} \ \ \ \ \ \ \ (k\geq 1,0 < r_i < l,0 < s_i < m,\ \text{where $1\leq i\leq k$}).
\]
Let $G$ be a generalised triangle group. In \cite{gen traingle groups}, a representation $\rho$ from $G$ to $PSL(2,\mathbb{C})$ was constructed, such that the orders of $a$, $b$ and $w$ in $PSL(2,\mathbb{C})$ are $l$, $m$ and $n$, respectively. The image of $G$ in $PSL(2,\mathbb{C})$ is finitely generated. By Mal'cev (\cite{malcev}), every finitely generated subgroup of $PSL(2,\mathbb{C})$ is residually finite. Hence, the hypotheses of \cref{lem: no collapse} are satisfied. Therefore, if $P=\langle a,b \mid a^l,b^m,w^n\rangle$, its residual deficiency is $2-1/l-1/m-1/n$, which is greater than one whenever $1/l+1/m+1/n<1$.
\vskip 2mm

In \cite{gen traingle groups} these groups were shown to have a finite index subgroup with deficiency greater than one whenever $1/l+1/m+1/n<1$. The arguments used are very similar to the arguments in \cref{lemma} and \cref{deficiency for finite index subgroups}. However, for part of the argument, the authors of \cite{gen traingle groups} rely on the fact that by removing the relator $w^n$ one ends up with a free product of finite cyclic groups, whereas \cref{lemma} and \cref{deficiency for finite index subgroups} do not.
\vskip 2mm

\end{ex}
\vskip 2mm

\begin{ex}\label{ex; one relator quot of free prod of cyclics}
A generalisation of the technique used in \cite{gen traingle groups} was used in \cite{Fine-Howie-Rosenberger} to prove that one-relator quotients of free products of non-trivial finite cyclic groups satisfy the Freiheitssatz. The type of groups dealt with in \cite{Fine-Howie-Rosenberger} are, more specifically, groups given by 
\begin{equation}\label{eq: one relator quot of free prod of cyclics}
P=\langle x_1,\ldots,x_n \mid x_1^{m_1},\ldots,x_n^{m_n}, w^{s}\rangle,
\end{equation}
where $m_i,s>1$, for all $i$, $1\leq i\leq n$, and $w$ is a word that involves all the generating elements $x_1,\ldots,x_n$. The Freiheitssatz condition in this case means that the subgroup generated by any $n-1$ elements from the generating set is the free product of the cyclic groups corresponding to the $n-1$ chosen elements. For instance, if the set of $n-1$ elements from the generating set consists of $x_1,\ldots,x_{n-1}$, then they generate a subgroup in $G$ isomorphic to $C_{m_1}\ast\cdots\ast C_{m_{n-1}}$. 
\vskip 2mm

This result was obtained by showing that $G$ admits a representation into $PSL(2,\mathbb{C})$ where the image of $x_i$ and the image of $w$ in $PSL(2,\mathbb{C})$ have order $m_i$ and $s$, respectively, for all $i$, $1\leq i\leq n$. Although not pointed out in \cite{Fine-Howie-Rosenberger}, this result is useful to prove that a substantial collection of these groups have a finite index subgroup with deficiency greater than one and hence are large. For this, it suffices to apply \cref{lem: no collapse}. The residual deficiency of $P$ is $n-\sum_{i=1}^{n}1/m_i-1/s$ and by \cref{deficiency for finite index subgroups}, whenever this quantity is greater than one, the group defined by $P$ will have a finite index subgroup with deficiency greater than one. The residual deficiency of $P$ is greater than one, for instance, for all $n\geq 4$, since $m_i,s>1$, for all $i$, $1\leq i\leq n$. If $n=3$, then one of $m_1,m_2,m_3$ or $s$ must be greater than two.
\vskip 2mm

\end{ex}

\vskip 2mm

\begin{ex}
A Coxeter group is a group $G$ given by a presentation of the type
\[
\langle a_1,\ldots,a_n  \mid a_1^{2},\ldots,a_n^{2},(a_ia_j)^{m_{ij}},1\leq i<j\leq n \rangle,
\]
where $m_{ij}\geq 2$. The value $m_{ij}$ may be $\infty$ in which case the word
$(a_ia_j)^{m_{ij}}$ is omitted from the set of relators.
\vskip 2mm

Since Coxeter groups are finitely generated and have faithful representations into $GL(n,\mathbb{R})$
(Appendix $D$ in \cite{davis}), they are residually finite (\cite{malcev}). Moreover, due to the solution of the word problem in Coxeter groups, the words $a_1,\ldots,a_n,a_ia_j,$ define non-trivial
elements of order $2$ and $m_{ij}$ in the group, respectively (\cite{davis}, chapter $3$, section
$4$). Hence, every Coxeter group $G$ can be surjected onto itself to obtain a residually finite quotient with no collapse. Applying \cref{lem: no collapse} gives that the residual deficiency of the presentation above for a Coxeter group is $n/2-\sum_{i<j}m_{ij}^{-1}$.

\vskip 2mm

It is known that Coxeter groups are either virtually abelian or large
(\cite{davis}, Theorem $14.1.2$). In fact, there is a way of telling,
given a Coxeter presentation, whether the group is large or not. For this, consider the \emph{Coxeter diagram} associated to the Coxeter group. Say the given Coxeter group is

\[
\langle a_1,\ldots,a_n  \mid a_1^{2},\ldots,a_n^{2},(a_ia_j)^{m_{ij}},1\leq i<j\leq n \rangle.
\]
The Coxeter diagram associated to this presentation is a graph with $n$
vertices where two vertices, say $(i,j)$, are joined by an edge if and
only if $m_{ij}\geq 3$. Consider the connected components of the graph.
These are called the irreducible components of the Coxeter diagram. These
components are Coxeter diagrams of other Coxeter groups. If all these
components correspond to either spherical Coxeter groups (i.e. finite
Coxeter groups), or Euclidean Coxeter groups, both of which have been
completely classified (\cite{davis}, Appendix $C$), then the group is
virtually abelian. However, if one of the components of the Coxeter diagram is
not either spherical or Euclidean, then the group is large.
\vskip 2mm

We already know that having a finite index subgroup with deficiency greater than one is a stronger property than largeness. Hence, this property may be used to study large Coxeter groups. There are plenty of examples of Coxeter groups with a finite index subgroup that has deficiency greater than one; given $n$, it suffices to have $n/2-\sum_{i<j}m_{ij}^{-1}>1$.
If, for instance, $m_{ij}>n(n-1)/(n-2)$, for all $m_{ij}$ which
appear as powers in the presentation, then a quick computation shows that
the residual deficiency is greater than one. On the other hand, the class
of Coxeter groups which are fundamental groups of $3$-dimensional hyperbolic
orbifolds, although large (by \cite{agol}), cannot have finite index subgroups with deficiency greater than one (\cite{epstein}).
\end{ex}

\vskip 2mm

\begin{ex}
Consider the presentation
\begin{equation}\label{eq: tetrahedral pres}
\langle x_1,x_2,x_3 \mid x_1^{e_1}, x_2^{e_2}, x_3^{e_3}, R_1^{m}, R_2^{p}, R_3^{q} \rangle,
\end{equation}
where $e_i=0$ or $e_i\geq 2$ for $i=1,2,3$; $2\leq m,p,q$; $R_1$ is a cyclically reduced word in the free product on $x_1, x_2$ which involves both $x_1$ and $x_2$ but not $x_3$, $R_2$ is a cyclically reduced word in the free product on $x_1, x_3$ which involves both $x_1$ and $x_3$ but not $x_2$ and $R_3$ is a cyclically reduced word in the free product on $x_2, x_3$ which involves both $x_2$ and $x_3$ but not $x_1$. Further, each $R_i$, $i=1,2,3$ is not a proper power in the free product on the generators it involves. 
\vskip 2mm

According to the authors of \cite{fine levin rosenberger}, E. Vinberg called a group \emph{generalised tetrahedral} if it is defined by a presentation such as \cref{eq: tetrahedral pres}. 
\vskip 2mm

Let $G$ be a generalised tetrahedral group. In \cite{fine levin rosenberger}, a representation $\rho: G\longrightarrow PSL(2,\mathbb{C})$ is constructed where $\rho(x_i)$ has infinite order if $e_i=0$ or exact order $e_i$ if $e_i\geq 2$ for $i=1,2,3$, and $\rho(R_1)$ has order $m$, $\rho(R_2)$ has order $p$, and $\rho(R_3)$ has order $q$. The authors of \cite{fine levin rosenberger} call such a representation \emph{essential}. In this case, an essential representation is the same as a quotient with no collapse. Therefore, by \cref{lem: no collapse}, the residual deficiency of $G$ is greater than one if
\begin{equation}\label{eq: tetra res def}
\dfrac{1}{e_1}+\dfrac{1}{e_2}+\dfrac{1}{e_3}+\dfrac{1}{m}+\dfrac{1}{p}+\dfrac{1}{q}<2,
\end{equation}
\vskip 2mm
where $1/e_i$ is $0$ in case $e_i=0$.
\vskip 2mm

The authors of \cite{fine levin rosenberger} use the representation $\rho$ to prove that if \cref{eq: tetra res def} holds then $G$ has a finite index subgroup with deficiency greater than one. They do so by using Lemma $3$ in \cite{gen traingle groups}, which is a particular case of \cref{lemma}. Lemma $3$ in \cite{gen traingle groups} uses the fact that when we remove $R_1^{m},R_2^{p}$ and $R_3^{q}$ from \cref{eq: tetrahedral pres}, we end up with a free product of finite cyclic groups. \cref{deficiency for finite index subgroups} does not use this fact.
\vskip 2mm

The authors of \cite{fine levin rosenberger} mention that their methods also give essential representations for groups defined by the following presentations:
\begin{equation}\label{eq: tetrahedral pres gen 1}
\langle x_1,\ldots,x_n \mid x_1^{e_1},\ldots, x_n^{e_n}, R_1^{m_1}(x_1,x_2),R_2^{m_2}(x_2,x_3),\ldots, R_n^{m_n}(x_n,x_1) \rangle,
\end{equation}
where $e_i=0$ or $e_i\geq 2$ for $i=1,\ldots,n$; $m_i\geq 3$ for $i=1,\ldots,n$, and for $i=1,\ldots,n$, $R_i(x_i,x_{i+1})$ is a cyclically reduced word in the free product on $x_i$, $x_{i+1}$ which involves both $x_i$ and $x_{i+1}$. And
\begin{equation}\label{eq: tetrahedral pres gen 2}
\langle x_1,\ldots,x_n  \mid x_1^{e_1}, \ldots, x_n^{e_n}, S_1^{m_1}(x_1,x_2),\ldots, S_{n-1}^{m_{n-1}}(x_1,x_n),S_n^{m_n}(x_2,x_3),
\end{equation}
\[
\ldots,S_{2n-3}^{m_{2n-3}}(x_2,x_n) \rangle,
\]
where $n\geq 3$, $e_i=0$ or $e_i\geq 2$ for $i=1,\ldots,n$, $m_j\geq 3$ for $j=1,\ldots,2n-3$, and for $j=1,\ldots,n-1$, $S_j$ is a cyclically reduced word in the free product on $x_1,x_j$ involving both $x_1$ and $x_j$. Moreover, for $j=0,\ldots,n-3$, $S_{n+j}$ is a cyclically reduced word in the free product on $x_2$ and $x_{j+3}$ involving both $x_2$ and $x_{j+3}$.
\vskip 2mm

The authors of \cite{fine levin rosenberger} do not mention whether the groups defined by \cref{eq: tetrahedral pres gen 1} and \cref{eq: tetrahedral pres gen 2} have a finite index subgroup with deficiency greater than one. However, the residual deficiency of these presentations is easily computed and hence by \cref{deficiency for finite index subgroups}, if the residual deficiency of \cref{eq: tetrahedral pres gen 1} or \cref{eq: tetrahedral pres gen 2} is greater than one, then the group it defines has a finite index subgroup with deficiency greater than one and hence are large.
\vskip 2mm

Some of the presentations above are relevant to the results in \cite{Edjvet-balanced}. In \cite{Edjvet-balanced}, M. Edjvet gave conditions under which balanced presentations define large groups. The balanced presentations considered needed to have at least two proper power relators. The proof was divided into two cases. The first case deals with balanced presentations which define groups with finite abelianisation. The second, when the presentations define groups with infinite abelianisation.
\vskip 2mm

If the group has finite abelianisation and at least one of the powers from the relators is greater than two, then the commutator subgroup has deficiency greater than one. If the group has infinite abelianisation and two of the powers in the relators are not coprime, then the group is large. This case uses results in \cite{Stohr} which only ensure largeness and not the existence of a finite index subgroup with deficiency greater than one.
\vskip 2mm

Consider the family of presentations given by \cref{eq: tetrahedral pres} and \cref{eq: tetrahedral pres gen 1}. Suppose $e_1=e_2=e_3=0$ for the presentations in \cref{eq: tetrahedral pres} and $e_i=0$, $i=1,\ldots,n$, for the presentations in \cref{eq: tetrahedral pres gen 1}. By the results in \cite{fine levin rosenberger} the presentations given by \cref{eq: tetrahedral pres} have residual deficiency greater than or equal to $3/2$, which is greater than one. Since $m_i\geq 3$, $1\leq i\leq n$, for the presentations given by \cref{eq: tetrahedral pres gen 1}, the residual deficiency is greater than or equal to $n-n/3$, which is greater than one if $n>1$. Therefore, a presentation in either of these two families which defines a group with infinite abelianisation, is an example of a group given by a balanced presentation which has a finite index subgroup with deficiency greater than one. Moreover, one does not need powers from the relators to have a common prime factor in order to establish largeness.
\vskip 2mm

It is not difficult to see that the families described above have many presentations that define groups with infinite abelianisation. It just suffices to have a single row or column of zeroes in its abelianisation matrix. We define the $(i,j)$ entry of the abelianisation matrix of a presentation $P = \langle x_1,\ldots, x_n \mid r_1,\ldots, r_m \rangle$ as the power sum of $x_j$ in $r_i$, which we denote by $\sigma_{x_j}(r_i)$. For example, for presentations given by \cref{eq: tetrahedral pres}, if $\sigma_{x_1}(R_1)=\sigma_{x_1}(R_2)=0$, then we get zeroes in the first column. For presentations given by \cref{eq: tetrahedral pres gen 1}, the abelianisation matrix only has non-zero elements on the diagonal or above it, with the exception of the $(n,1)$ entry corresponding to $\sigma_{x_1}(R_n^{m_n})$. If $\sigma_{x_1}(R_n)=0$ and $\sigma_{x_i}(R_i)=0$ for some $i$, where $i=1,\ldots,n$, then the abelianisation matrix reduces down to a upper triangular matrix with at least one zero in its diagonal. The determinant of such matrix is therefore $0$. Since the absolute value of the determinant of such matrix gives the index of the commutator subgroup in the group (index $0$ corresponding to infinite index), then the group defined by such presentation has infinite abelianisation.
\end{ex}

\vskip 2mm

\begin{ex}
Consider $m$ non-trivial elements $u_1,\ldots,u_m$ in $F_n$, the non-abelian free group of rank $n$. Denote by $\{a_1,\ldots,a_{n-1},t\}$, the canonical free generators of $F_n$. Suppose $\sigma_t(u_i)=s_i\neq 0$, for $1\leq i\leq m$. Define $k=|s_1\cdots s_m|$ and $n_i=|s_1\cdots s_{i-1}s_{i+1}\cdots s_m|$. Then the group $G$ given by the presentation
\[
P=\langle a_1,\ldots,a_{n-1},t \mid u_1^{n_1},\ldots, u_m^{n_m}\rangle,
\]
surjects onto $C_k$, the cyclic group of order $k$, by sending $a_i$ to $0$, for all $i$, $1\leq i\leq n-1$, and $t$ to $1$. The residual
deficiency of $G$ is bounded as follows
\[
rdef(G)\geq rdef(P)\geq n-\sum^{m}_{i=1}\dfrac{1}{n_i}.
\]
The above expression is greater than one if the values $|s_i|$ are sufficiently large. This construction gives examples of finitely presented groups of arbitrarily large negative deficiency which have finite index subgroups with deficiency greater than one. 
\vskip 2mm

\vskip 2mm

\end{ex}
\vskip 2mm

\begin{ex}
Consider the following deficiency one presentation where $n\geq 2$
\[
P=\langle a_1,\ldots,a_{n} \mid u_1^{s_1},\ldots, u_{n-1}^{s_{n-1}}\rangle,
\]
such that there is a prime $p$ that divides all the powers $s_i$, where $u_i$ is a non-trivial word in $F_n$, for all $i$, $1\leq i\leq n-1$.
As $F_n$ is residually finite $p$, then there is a normal
subgroup $H$ in $F_n$, with index in $F_n$ a power of $p$, such
that $u_{i_{0}}\notin H$, for some $i_0$, $1\leq i_0\leq n-1$, and
has minimal index with respect to this property. This means that
if $K$ is a finite index normal subgroup of $F_n$ with index a power of $p$
strictly less than $|F_n:H|$, then $u_i\in K$ for all $i$, $1\leq
i\leq n-1$.

\vskip 2mm

As the quotient $F_n/H$ is a $p$-group, it has non-trivial centre. Take a subgroup $M$ of order $p$ in the centre. Denote by $L$
the pullback of $M$ in $F_n$ under the canonical surjective homomorphism from $F_n$ to $F_n/H$. Note that $L$ is normal in $F_n$. As the index of $L$ in $F_n$ is less than the index of $H$ in $F_n$, then $u_i\in L$, $1\leq i\leq n-1$. This means that $u_{i_0}$ and any other $u_i$ not in $H$, has order $p$ in
$F_n/H$. Hence the residual deficiency of $P$ is bounded as follows
\[
rdef(P)\geq n-(n-2)-\dfrac{1}{p}>1.
\]
Therefore, the group defined by $P$ has a finite index subgroup with deficiency greater than one. Note that since $p$ divides all powers $s_i$, then the group has $p$-deficiency greater than one and hence it is large (\cite{jack-anitha}). However, this result does not tell us that the group has a finite index subgroup with deficiency greater than one.
\vskip 2mm

It is known by \cite{Stohr}, that all deficiency one presentations with a proper power relator define large groups. However, the question of whether such presentations define groups which have finite index subgroups with deficiency greater than one is still open. The previous example shows that if the deficiency one presentation only has proper power relators and the powers are all divided by a common prime, then the group defined has a finite index subgroup with deficiency greater than one. In particular, if the presentation has two generators and one relator which is a proper power, then the group it defines has a finite index subgroup with deficiency greater than one.
\end{ex}
\vskip 2mm

\section{Relative size of subgroups and the supermultiplicativity of residual deficiency}

The definitions and results in this section follow similar lines to the ones in Section $2$ of \cite{puchta-yiftach}.
\vskip 2mm

\begin{defn}\label{def: relsize}
Let $G$ be a group, $K$ a normal subgroup of $G$ and $g$ an element in $G$. Define the \emph{relative size of $g$ in $G$ with respect to $K$} as follows.
\vskip 2mm

First suppose $g\in K$. Define $\nu(g;G,K)$ to be the supremum over $\{o(\psi(a),G/K)\}$, where $a\in G$ is such that $a^n=g$ for some integer $n$, and where $\psi:G\longrightarrow G/K$. 
\vskip 2mm

If $g\notin K$, consider all possible expressions
of $g$ as $a^{n}$ where $n$ is a non-zero
integer. In this case, define $\nu(g;G,K)$ as the supremum over all
integers arising in such expressions. 
\vskip 2mm

The relative size of $g$ in $G$ with respect to $K$ is defined as $\nu(g;G,K)^{-1}$.
\vskip 2mm

\begin{note}
We say that the value $\nu(g;G,K)$ is infinite if the set $\{o(\psi(a),G/K)\}$ is not bounded. That is, if for every $\varepsilon>0$, there is $a\in G$ such that $a^{n}=g$, for some integer $n$, and $o(\psi(a),G/K)\geq \varepsilon$. In this case, the relative size of $g$ in $G$ with respect to $K$ is defined as $0$. 
\end{note}
\vskip 2mm

Suppose $\nu(g;G,K)$ is finite. We call an element $a\in G$ which reaches the supremum in either of the conditions stated above, a \emph{minimal root of $g$ in $G$ with respect to $K$}.
\end{defn}
\vskip 2mm

\begin{defn}\label{def: rel size of S in G wrt K}
Let $K$ be a normal subgroup in $G$ and $S=\{s_1,\ldots,s_m\}\subseteq G$. 
We define $relsize(S;G,K)$ the \emph{relative size of $S$ in $G$ with respect to $K$} to be
\[
relsize(S;G,K)=\sum_{i=1}^{m}\nu(s_i;G,K)^{-1}.
\]
Let $M$ be a finitely generated normal subgroup of $G$. We define $relsize(M;G,K)$ the \emph{relative size of
$M$ in $G$ with respect to $K$} to be
\[
relsize(M;G,K)=\underset{\underset{\langle\langle S\rangle\rangle^{G}=M}{\underset{|S|<\infty}{S\subset
G}}}{\inf}\{relsize(S;G,K)\}.
\]
\end{defn}
\vskip 2mm

\begin{rem}
Note that \cref{def: residual def} may be restated in the following way. Let $G$ and $Q$ be as in \cref{def: residual def}. Denote by $\overline{R}_G$ the inverse image of $R_G$ under $\varphi: F_n\longrightarrow G$. Since $R\leqslant \overline{R}_G$, then
\[
rdef(Q)=n-relsize(R;F_n,\overline{R}_G),
\]
\[
rdef(G)=\underset{\langle X|R\rangle\cong G}{\text{sup}} \left\{ |X|-relsize(R;F_n,\overline{R}_G) \right\},
\]
where $F_n$ is freely generated by $X$ and the supremum is taken over all finite presentations of $G$. Note that for $\omega\in F_n$ and $N$ a normal subgroup in $F_n$, $\nu(\omega;F_n,N)$ is always finite. 
\end{rem}

\vskip 2mm

\begin{lem}\label{new lemma}
Let $G$ be a group, $H$ a finite index normal subgroup in $G$, $K$ a normal subgroup in $G$ and $g$ an element in $H$. If $\nu(g;G,K)$ is infinite, then $\nu(g;H,H\cap K)$ is infinite too.
\end{lem}
\begin{proof}
First suppose $g\in K$. Denote by $\psi$ the canonical surjective homomorphism from $G$ onto $G/K$. If $\nu(g;G,K)$ is infinite, then for all $n\in \mathbb{N}$ there is an element $a\in G$ such that $o(\psi(a),G/K)=m$, where $m\geq n$. 
\vskip 2mm

Let $s=|G:H|$ and suppose $\nu(g;H,H\cap K)=C$. Consider $s(C+1)=n$. Since, $\nu(g;G,K)$ is infinite, there is $a\in G$ such that $o(\psi(a),G/K)=m\geq s(C+1)$. As $s=|G:H|$, then $a^{s}\in H$. Since $o(\psi(a),G/K)=m\geq s(C+1)$, then $o(\psi(a^s),H/K)\geq C+1$ which contradicts the assumption that $\nu(g;H,H\cap K)=C$.
\vskip 2mm

If $g\notin K$, then instead of considering $o(\psi(a),G/K)$, we consider $n\in \mathbb{N}$, such that $a^n=g$. In this case, the arguments above work equally well.
\end{proof}
\vskip 2mm

Equivalently, this lemma says that if the relative size of $g$ in $G$ with respect to $K$ is zero, then its relative size in $H$ with respect to $H\cap K$ is also zero.
\vskip 2mm

\begin{lem}\label{lem: prev to supermult of res def: alternative}
Let $G$ be a group. Consider $H$ a finite index normal subgroup
of $G$ and $K$ a normal subgroup of $G$. Let $g$ be an element of $H$ and suppose $g$ has a minimal root $a$ in $G$ with respect to $K$. Denote by
$\psi$ the surjective homomorphism from $G$ to $G/H$. Then
\[
\nu(g;H,K\cap H)\geq \dfrac{\nu(g;G,K)}{o(\psi(a), G/H)}.
\]
\end{lem}
\begin{proof}
Let $l=o(\psi(a), G/H)$. First assume $g\notin K$ and $m=\nu(g;G,K)$. In this case $a^{m}=g$. As $g\in H$ and $a^l\in H$, then $l$ divides $m$. Since $(a^l)^{m/l}=a^m=g$, then $\nu(g;H,K\cap H)\geq m/l$.
\vskip 2mm

Now suppose $g\in K$. Let $m=\nu(g;G,K)$. By definition of $\nu(g;G,K)$, $a^{m}\in K$ and $a^{j}\notin K$ for all $j<m$. If $l=o(\psi(a), G/H)$, then $a^{l}\in H$. Therefore, if $n$ is the minimum common multiple of $l$ and $m$, then $a^{n}\in H\cap K$. Since $n$ is a multiple of $l$, then $l$ divides $n$ and hence $a^n=(a^l)^{n/l}$. 
\vskip 2mm

The element $a^l$ belongs to $H$. Denote it by $b$. Then $b^{n/l}\in K$ and $b^i\notin K$ for all $i<n/l$. Since $g$ and $b$ are elements in $H$, then $\nu(g;H,K\cap H)\geq n/l\geq m/l$, where the last inequality follows from the fact that $n$ is a multiple of $m$.
\end{proof}
\vskip 2mm

\begin{thm}\label{thm: res size inequality: alternative}
Let $M$ and $K$ be normal subgroups of $G$, where $M$ is finitely generated as a normal subgroup of $G$. If $H$ is a finite index
normal subgroup of $G$ which contains $M$, then
\[
relsize(M;H,K\cap H)\leq |G:H|relsize(M;G,K).
\]

\end{thm}
\begin{proof}
Let $S=\{s_1,\ldots,s_d\}\subseteq G$ generate $M$ as a normal subgroup in $G$. First suppose that the relative size of $s_i$ in $G$ with respect to $K$ is greater than zero for all $i$, $1\leq i\leq d$. Hence, consider the set $A=\{a_1,\ldots,a_d\}$, where $a_i$ is a minimal
root of $s_i$ in $G$ with respect to $K$. Denote the centraliser of $s_i$ in $G$ by
$C_{G}(s_i)$. Let $k_i=|C_{G}(s_i):C_H(s_i)|$ and
$l_i=o(\psi(a_i),G/H)$, where $\psi$  is the canonical surjective
map from $G$ to $G/H$. 
\vskip 2mm

Let $n_i\in \mathbb{N}$ be such that $a_i^{n_i}=s_i$, where $s_i\in H$. Then for all $t\in \mathbb{Z}$, $a_i^{t}$ commutes with $s_i$ and hence the subgroup in $G$ generated by $a_i$ is in $C_G(s_i)$ for all $i$, $1\leq i\leq d$. Moreover, since $l_i=o(\psi(a_i),G/H)$, then $a_i^{l_i}\in H\cap C_G(s_i)$ and $a_i^r\notin C_H(s_i)$ for all $r$, $1\leq r< l_i$. This means that if $r$ and $r'$ are integers such that $1\leq r,r'<l_i$, with $r\neq r'$, then $a^r$ and $a^{r'}$ cannot be in the same coset of $C_H(s_i)$ in $C_G(s_i)$. Therefore, for each $i$, $1\leq i\leq d$, there are at most $l_i$ distinct cosets of $C_H(s_i)$ in $C_G(s_i)$ which contain an element of the form $a_i^{t}$, where $1\leq t\leq l_i$. Hence, $l_i\leq k_i$, for all $i$, $1\leq i\leq d$.
\vskip 2mm

Obtain, as in \cref{lemma}, a set $S'=\{s_{ij}\}$ of generators for $M$ as a normal subgroup of
$H$, where $i=1,\ldots,d$. Note that for each $i$ the index $j$ ranges from $1$ to $m/k_i$, where $m=|G:H|$. 
\vskip 2mm

For each $i$, the set $\{s_{ij}\}$ is obtained by conjugating $s_i$ by the elements in a transversal of $C_G(s_i)$ in $G$. Therefore, for each $i$, $s_{ij}$ and $s_{ij'}$ are conjugate to one another in $G$. Since both $K$ and $H$ are normal in $G$, the order of the image of $a_i$ in $G/H$ and the order of the image of $a_i$ in $G/K$ are invariant under conjugation. Moreover, if $a_i$ is a minimal root of $s_i$ with respect to $K$ and $b$ is an element in a transversal of $C_G(s_i)$ in $G$, then $ba_ib^{-1}$ is a minimal root of $bs_ib^{-1}$ with respect to $K$. Therefore, for each $i$, $\nu(s_{ij};H,K\cap H)=\nu(s_{ij'};H,K\cap H)$ for all $j$ and $j'$ such that $1\leq j,j'\leq m/k_i$. From this and \cref{lem: prev to supermult of res def: alternative} the following holds
\[
relsize(S';H,K\cap H)=\sum_{i=1}^{d}\sum_{j=1}^{m/k_i}\nu(s_{ij};H,K\cap H)^{-1}=\sum_{i=1}^{d}\dfrac{|G:H|}{k_i}\nu(s_i;H,K\cap H)^{-1}\leq
\]
\begin{equation}\label{eq: thm rel size ineq}
\sum_{i=1}^{d}\dfrac{|G:H|}{k_i}l_i\nu(s_i;G,K)^{-1}\leq
|G:H|\sum_{i=1}^{d}\nu(s_i;G,K)^{-1}= 
\end{equation}

\[
|G:H|relsize(S;G,K).
\]
\vskip 1mm

Hence, if $S$ generates $M$ as a normal subgroup of $G$, then $S'$ generates $M$ as a normal
subgroup of $H$, and $relsize(S';H,K\cap H)\leq |G:H|relsize(S;G,K)$.
Therefore,
\[
relsize(M;H,K\cap H)\leq relsize(S';H,K\cap H)\leq |G:H|relsize(S;G,K),
\]
for all finite generating sets $S$. The result is then obtained by
considering the infimum over all finite generating sets $S$ on the right hand side of the
previous inequalities.
\vskip 2mm

Now let us consider the situation when $s_i$ has zero relative size in $G$ with respect to $K$ for some $i$, $1\leq i\leq d$. By \cref{new lemma} it has zero relative size in $H$ with respect to $H\cap K$. Moreover, as $H$ and $K$ are normal in $G$, and order is preserved under conjugation, then any conjugate of $s_i$ by an element in $G$ also has zero relative size in $G$ and $H$ with respect to $K$. Therefore, in \cref{eq: thm rel size ineq}, if $s_i$ has zero relative size in $H$ with respect to $K$, then $s_{ij}$, where $1\leq j\leq m/k_i$, also has zero relative size in $H$ with respect to $H\cap K$, and hence they do not need to be considered in the computation.

\end{proof}
\vskip 2mm

We now prove that residual deficiency minus one is supermultiplicative.
\vskip 2mm

\begin{thm}\label{thm: supermult rdef}
Let $G$ be a finitely presented group and $H$ a finite index
normal subgroup of $G$. Then
\begin{equation}\label{eq: thm supermult rdef}
rdef(G)-1\leq \dfrac{rdef(H)-1}{|G:H|}.
\end{equation}
\end{thm}
\begin{proof}
Let $P=\langle X|T\rangle$ be a finite presentation for $G$, where $n=|X|$, and let $Q=\langle Y|S\rangle$
be the presentation for $H$ obtained by applying the Reidemeister-Schreier rewriting process to $P$. Suppose $Y$ freely generates $F_k$. Note $k=|G:H|(n-1)+1$. By \cref{eq: thm rel size ineq} in
\cref{thm: res size inequality: alternative}
\[
relsize(S;F_k,\overline{R}_G)\leq |G:H|relsize(T;F_n,\overline{R}_G).
\]
Note that $\overline{R}_G\leqslant F_k$ in the previous inequality.
Therefore,
\[
|G:H|(rdef(P)-1)=|G:H|(n-relsize(T;F_n,\overline{R}_G)-1)=
\]
\[
|G:H|(n-1)-|G:H|relsize(T;F_n,\overline{R}_G)\leq |G:H|(n-1)-relsize(S;F_k,\overline{R}_G)=
\]
\[
|G:H|(n-1)+1-relsize(S;F_k,\overline{R}_G)-1= rdef(Q)-1\leq rdef(H)-1.
\]
The result follows by taking the supremum over all finite
presentations $P$ of $G$ on the left hand side of the previous
inequalities.
\end{proof}

By \cref{thm: supermult rdef} the residual deficiency of a group is supermultiplicative on finite index
normal subgroups. Therefore, following the spirit in Section $1$, we define the \emph{residual deficiency gradient} of $G$ by
\[
RDG(G)=\underset{H\underset{f}{\trianglelefteq}G }{\text{sup}}
\left\{ \dfrac{rdef(H)-1}{|G:H|} \right\},
\]
which by \cref{prop: 1 section 2} gives a multiplicative invariant and hence a generalised Euler characteristic.

\vskip 2mm

However, by \cref{cor positive res def}, the residual deficiency gradient is the same as the deficiency gradient. One inequality is straightforward; it suffices to notice that for every finitely presented group $G$, $def(G)\leq rdef(G)$. This then implies that $DG(G)\leq RDG(G)$. By \cref{cor positive res def}, 
$rdef(G)-1\leq DG(G)$. Therefore, if $N$ is a finite index normal subgroup of $G$, then
\[
\dfrac{rdef(N)-1}{|G:N|}\leq \dfrac{DG(N)}{|G:N|}=DG(G).
\]
Taking the supremum on the left hand side, over all finite index normal subgroups of $G$, gives $RDG(G)\leq DG(G)$.

\section{The residual deficiency of quotients}

The following two results give lower bounds for the residual deficiency of quotients of a finitely presented group $G$.

\begin{prop}\label{prop: rdef qt fin order}
Let $G$ be a finitely presented group with non-trivial finite residual $R_G$. Say $s_1,\ldots,s_m$ are non-trivial elements in $G$ contained in $R_G$. Denote by $H$ the normal subgroup in $G$ generated by $S=\{s_1,\ldots,s_m\}$. Then
\[
rdef(G/H)\geq rdef(G)-relsize(S;G,R_G)
\]
\end{prop}
\begin{proof}
Let $P=\langle X|R\rangle$ be a finite presentation for $G$, where $X$
freely generates $F_d$, the non-abelian free group of rank $d$. First, separate $S$ into two disjoint subsets $S_1$ and $S_2$, where the elements in $S_1=\{s_1\ldots,s_k\}$ have relative size in $G$ with respect to $R_G$ greater than zero, and the ones in $S_2=\{s_{k+1},\ldots,s_m\}$ do not. Let $a_1,\ldots,a_k$ be minimal roots in $G$ with respect to $R_G$ for $s_1,\ldots,s_k$, respectively. Moreover, let $q_1,\ldots,q_k\in\mathbb{N}$ be such that $a_i^{q_i}=s_i$, for $1\leq i\leq k$.
\vskip 2mm

Consider $w_1,\ldots,w_k\in F_d$ such that $\varphi(w_1)=a_1,\ldots,\varphi(w_k)=a_k$, for all $1\leq i\leq k$, where $\varphi$ is the canonical map going from $F_d$ to $G$. If $w_i=u_i^{r}$, for $r>1$, then $\varphi(u_i)^{r}=\varphi(u_i^{r})=\varphi(w_i)=a_i$, which would imply that $a_i$ is not a minimal root for $s_i$. Therefore, $w_i$ cannot be expressed as a proper power of any other element in $F_d$. Hence, if $W_1=\{ w_1^{q_1},\ldots,w_k^{q_k} \}$, then
\[
relsize(S_1;G,R_G)=relsize(W_1;F_d,\overline{R}_G).
\]
\vskip 2mm

If $Q$ is a finite presentation for $G/H$, then by the definition of residual deficiency, we have $rdef(Q)\leq rdef(G/H)$. We claim that for every $\varepsilon>0$, there is a finite presentation $Q_{\varepsilon}$ for $G/H$, such that
\[
rdef(Q_{\varepsilon})+\varepsilon> rdef(P)-relsize(S;G,R_G).
\]
This would then imply
\[
rdef(G/H)+\varepsilon> rdef(P)-relsize(S;G,R_G),
\]
for all $\varepsilon>0$, which means
\[
rdef(G/H)\geq rdef(P)-relsize(S;G,R_G),
\]
for all finite presentations $P$ of $G$. The result would then be obtained by taking the supremum over all finite presentations for $G$.

\vskip 2mm

Let $a_{k+1},\ldots,a_m\in G$ be such that $o(\psi(a_j),G/R_G)>(m-k)\varepsilon^{-1}$, for $k+1\leq j\leq m$, where $\psi:G\longrightarrow G/R_G$. Such $a_j$ exist since $s_{k+1},\ldots,s_m$ have zero relative size in $G$ with respect to $R_G$.
\vskip 2mm

Suppose $a_j^{q_j}=s_j$, for $k+1\leq j\leq m$. Let $W_2=\{w_{k+1}^{q_{k+1}},\ldots,w_m^{q_m}\}$ be a subset of $F_d$ such that $\varphi(w_j)=a_j$. Then $\varphi(w_j^{q_j})=\varphi(w_j)^{q_j}=s_j$, for $k+1\leq j\leq m$. Hence, $Q_{\varepsilon}=\langle X\mid R,W_1,W_2  \rangle$ is a finite presentation for $G/H$.
\vskip 2mm

First note that
\[
rdef(Q_{\varepsilon})=rdef(P)-relsize(W_1\cup W_2;F_d,\overline{R}_G).
\]
Second, that
\[
relsize(W_1\cup W_2;F_d,\overline{R}_G)=relsize(W_1;F_d,\overline{R}_G)+relsize(W_2;F_d,\overline{R}_G).
\]
We already have that $relsize(W_1;F_d,\overline{R}_G)=relsize(S_1;G,R_G)$. Moreover, given how we chose $a_j$ for $k+1\leq j\leq m$, we have
\[
relsize(W_2;F_d,\overline{R}_G)\leq \sum_{j=k+1}^{m}o(\psi(a_j),G/R_G)^{-1}<\sum_{j=k+1}^{m}\dfrac{\varepsilon}{m-k}=\varepsilon.
\]
Therefore,
\[
rdef(Q_{\varepsilon})+\varepsilon=rdef(P)+\varepsilon-relsize(S_1;G,R_G)-relsize(W_2;F_d,\overline{R}_G)>
\]
\[
rdef(P)+\varepsilon-relsize(S_1;G,R_G)-\varepsilon.
\]
Since 
\[
relsize(S;G,R_G)=relsize(S_1,G,R_G)+relsize(S_2,G,R_G),
\]
and $relsize(S_2,G,R_G)=0$, we obtain
\[
rdef(Q_{\varepsilon})+\varepsilon >rdef(P)-relsize(S;G,R_G).
\]
\end{proof}
\vskip 2mm

\begin{thm}\label{thm: rdef qt inf order}
Let $G$ be a finitely presented group. Let $g_1,\ldots,g_m$ be
non-trivial elements in $G$ with infinite order in $G/R_G$. Denote by $H_n$
the normal subgroup in $G$ generated by $g_1^{n},\ldots,g_m^{n}$. Then given $\varepsilon>0$, there exists a natural
number $k$, such that for all $q\in \mathbb{Z}\backslash\{0\}$, we have
\[
rdef(G/H_{kq})\geq rdef(G)-\varepsilon.
\]
\end{thm}
\begin{proof}
Let $P=\langle X|R\rangle$ be a finite presentation for $G$, where $X$ generates $F_d$ the non-abelian free group of rank $d$. Let $R=\{u_1^{a_1},\ldots,u_r^{a_r}\}$, where $u_i$ cannot be expressed as a proper power of any other element in $F_d$, for all $i$, $1\leq i\leq r$. Denote by $l_i$ the order of $u_i$ in $F_d/\overline{R}_G$. 
\vskip 2mm

As $F_d/\overline{R}_G$ is residually finite, there is a finite index normal subgroup $K$ of $F_d$, which contains $\overline{R}_G$, such that the order of $u_i$ in $F_d/K$ is $l_i$, for all $i$, $1\leq i\leq r$.

\vskip 2mm

Let $T$ be a positive natural number. Denote by $w_j$ an element in $F_d$ which corresponds to $g_j$ under $\varphi$, the canonical homomorphism from $F_d$ to $G$. As the order of $g_j$ in $G/R_G$ is infinite, so is the order of $w_j$ in $F_d/\overline{R}_G$. Therefore, there is a finite index normal subgroup $L_j$ of $F_d$, which contains $\overline{R}_G$, such that $w_j^{n}$ is not in $L_j$ for all $n\leq T$. Note that this can be done for all $j$, $1\leq j\leq m$. 
\vskip 2mm

Take $M$ the intersection of $K$ and $\bigcap_{j=1}^{m} L_j$. The subgroup $M$ is a finite index normal subgroup of $F_d$ and contains $\overline{R}_G$. Also, the order of $u_i$ in $F_d/M$ is $l_i$, for all $i$, $1\leq i\leq r$. Consider $k_j$ the order of $w_j$ in $F_d/M$. Then, for every multiple $q$ of $k_j$, $w_j^{q}\in M$. Therefore, if $k$ is the minimum common multiple of $k_1,\ldots,k_m$, then $w_j^k\in M$ for all $j$, $1\leq j\leq m$.
\vskip 2mm

Consider $Q=\langle X\mid R,w_1^{kq},\ldots,w_m^{kq}\rangle$, where $q\in \mathbb{Z}\backslash \{0\}$. The presentation $Q$ is a finite presentation for the group $G/H_{kq}$. If $W=\{w_1^{kq},\ldots,w_m^{kq}\}$, then since $T\leq k_j$ for all $j$, $1\leq j\leq m$, 
\[
relsize(W;F_d,\overline{R}_{G/H_{qk}})\leq \sum_{j=1}^{m}\dfrac{1}{k_j}\leq\sum_{j=1}^{m}\dfrac{1}{T}. 
\]
As $rdef(Q)\geq rdef(P)-relsize(W;F_d,\overline{R}_{G/H_{qk}})$, then 
\[
rdef(G/H_{kq})\geq rdef(Q)\geq rdef(P)-\sum_{j=1}^{m}\dfrac{1}{T}=rdef(P)-\dfrac{m}{T}.
\]
The natural number $T$ was chosen arbitrarily. Therefore, take $T$ such that $m/T<\varepsilon$. Then $rdef(P)-m/T>rdef(P)-\varepsilon$, and hence
\[
rdef(G/H_{kq})>rdef(P)-\varepsilon.
\]
The result follows by taking the supremum over all finite presentations $P$ of $G$.
\end{proof}

\vskip 2mm

\begin{cor}
Let $G$ be a finitely presented residually finite group. Let
$g_1,\ldots,g_m$ be non-trivial elements in $G$. Denote by $H_k$ the normal group in $G$ generated by
$g_1^{k},\ldots,g_m^{k}$. If $G$ has residual deficiency greater than one,
then $G/H_{k}$ has residual deficiency greater than one for an infinite collection of natural numbers $k$.
\end{cor}
\begin{proof}
Let $g_1,\ldots,g_j$ have finite order in $G$ while $ g_{j+1}\ldots,g_m $ infinite order. Denote by $l_1,\ldots,l_j$ the orders of $g_1,\ldots,g_j$, respectively. Let $l$ be the minimum common multiple of $l_1,\ldots,l_j$. Therefore, $g_1^{l},\ldots,g_j^{l}$ are all trivial in $G$.
\vskip 2mm

Since $g_{j+1},\ldots,g_m$ have infinite order in $G$, then by \cref{thm: rdef qt inf order} there is a $k$ such that $G/\langle\langle  g_{j+1}^{kq},\ldots,g_m^{kq}\rangle\rangle$ has residual deficiency greater than one for all naturals $q\in \mathbb{N}$.
\vskip 2mm

Take $k$ to be a multiple of $l$. Note that as $g_1^{k},\ldots,g_j^{k}$ are trivial in $G$, then as normal subgroups in $G$ we have $\langle\langle g_{j+1}^{kq},\ldots,g_m^{kq}\rangle\rangle=\langle\langle g_1^{kq},\ldots,g_m^{kq}\rangle\rangle$. The result then follows.

\end{proof}

\vskip 2mm

If the residually finite group is a non-abelian free group of finite rank, then the following lemma helps us say something more.
\vskip 2mm

\begin{lem}$(\cite{olshanskii-osin}$, Lemma 2.4$)$\label{lem: olshanski osin}\\
For any finite collection of non-trivial elements $g_1,\ldots,g_m$ of a 
non-abelian free group $F$ and any number $k\in \mathbb{N}$, there exists 
$K\in\mathbb{N}$ with the following property. For every $q\geq K$, there 
is a finite index normal subgroup $N_q\triangleleft F$ such that for all 
$1\leq i\leq m$, $g_i^s\notin N_q$ whenever $1\leq s\leq k$, but $g_i^{q}\in N_q$.
\end{lem}
\vskip 2mm

Consider $k=m$. By \cref{lem: olshanski osin}, there is an $M\in\mathbb{N}$ 
such that for all $q\geq M$ there is a finite index normal subgroup $N_q$ in $F$ 
such that for all $i$, $1\leq i\leq m$, $g_i^s\notin N_q$ whenever $1\leq s\leq m$, but $g_i^{q}\in N_q$. Consider $F_d$ the non-abelian free group of rank $d$. Then, the residual deficiency of $F_d/\langle\langle g_1^q,\ldots,g_m^q \rangle\rangle$ is greater than or equal to $d-\sum_{i=1}^{m}1/(m+1)$, which is greater than one. Hence,
\vskip 2mm

\begin{prop}\label{thm: res def greater than one for qts of free grps}
Let $F_d$ be a non-abelian free group of rank $d\geq 2$, and let $g_1\ldots,g_k$ be arbitrary elements of $F_d$. Then $G=F_d/\langle\langle g_1^q,\ldots,g_m^q  \rangle\rangle$ has residual deficiency greater than one for all but finitely many $q\in\mathbb{N}$.
\end{prop}
\vskip 2mm

The authors of \cite{olshanskii-osin} used \cref{lem: olshanski osin} to prove that for all but finitely many $q\in \mathbb{N}$, the group $G=F_r/\langle\langle g_1^q,\ldots,g_m^q  \rangle\rangle$ has a finite index normal subgroup with deficiency greater than one. \cref{ex: resdef-1 < DG} shows that having a finite index normal subgroup with deficiency greater than one does not imply the group has residual deficiency greater than one. Therefore \cref{thm: res def greater than one for qts of free grps} says something stronger about $G=F_r/\langle\langle g_1^q,\ldots,g_m^q  \rangle\rangle$.
\vskip 2mm

\cref{prop: rdef qt fin order} and \cref{thm: rdef qt inf order} give lower bounds for the residual deficiency of a quotient $G/M$, where $M$ is finitely generated. In some cases, the lower bounds are such that if the residual deficiency of $G$ is greater than one, then the residual deficiency of $G/M$ is also greater than one. The following discussion examines a case when the residual deficiency of $G/M$ is not greater than one. The aim is to find conditions under which $G/M$ has a finite index subgroup with residual deficiency greater than one.
\vskip 2mm

\cref{thm: res size inequality: alternative} says that
\[
relsize(M;H,K\cap H)\leq |G:H|relsize(M;G,K),
\]
where $G$ is a finitely presented group, $H$ a finite index normal subgroup of $G$, $K$ a normal subgroup of $G$ and $M$ a finitely generated normal subgroup of $G$ contained in $H$. The proof in \cref{thm: res size inequality: alternative} first shows that if $S=\{ s_1, \ldots,s_d \}$ is a generating set for $M$ as a normal subgroup of $G$, then
\begin{equation}\label{eq: submult of relsize}
relsize(S';H,K\cap H)\leq |G:H|relsize(S;G,K),
\end{equation}
where $S'$ is a generating set for $M$ as a normal subgroup in $H$, as obtained in \cref{lemma}. However, from \cref{eq: thm rel size ineq}, the previous inequality is strict when either $\nu(s_i;H,K\cap H)^{-1}<l_i\nu(s_i;G,K)^{-1}$, or $l_i/k_i<1$, for some $i$, $1\leq i\leq d$. Finding conditions when $\nu(s_i;H,K\cap H)^{-1}<l_i\nu(s_i;G,K)^{-1}$ has proved to be difficult. However, there are conditions that imply $l_i/k_i<1$. 
\vskip 2mm

Before specifying these conditions, let us examine the utility of having a strict inequality in \cref{eq: submult of relsize}. 
\vskip 2mm

Suppose $G$ is a finitely presented non-residually finite group. Let $M$ be a finitely generated normal subgroup of $G$, contained in $R_G$, normally generated in $G$ by $S=\{ s_1, \ldots,s_d \}$. Moreover, assume $rdef(G)-relsize(S;G,R_G)\geq 1$. It is worth noting that by \cref{prop: rdef qt fin order}
\begin{equation}\label{eq: exploiting submult of relsize}
1\leq rdef(G)-relsize(S;G,R_G)\leq rdef(G/M).
\end{equation}
The residual deficiency of $G/M$ may be greater than one. However, if there is no evident way of verifying it and $l_i/k_i<1$, then \cref{eq: thm supermult rdef} and a strict inequality in \cref{eq: submult of relsize} give
\[
0\leq |G:H|\big(rdef(G)-relsize(S;G,R_G)-1\big)<rdef(H)-relsize(S';H,R_G)-1.
\]
Hence,
\[
1<rdef(H)-relsize(S';H,R_G)\leq rdef(H/M),
\]
where the last inequality is a consequence of \cref{prop: rdef qt fin order}.
\vskip 2mm

Therefore, if we have a strict inequality in \cref{eq: submult of relsize}, the subgroup $H/M$, which is a finite index normal subgroup of $G/M$, has residual deficiency greater than one. In particular, all properties implied by having residual deficiency greater than one, which are inherited under finite index supergroups, are enjoyed by $G/M$.
\vskip 2mm

The following proposition gives conditions under which \cref{eq: submult of relsize} is a strict inequality.
\vskip 2mm

\begin{prop}\label{thm: resdef equal to one}
Let $G$ be a non-residually finite, finitely presented group, such that $rdef(G)=n>1$, where $n\in \mathbb{N}$. Let $W=\{w_1,\ldots,w_{n-1}\}$ be a set of non-trivial elements in $R_G$. Suppose $w_i$, for some $i$, $1\leq i\leq n-1$, is such that there is a non-trivial element $a\in C_G(w_i)-R_G$. Then, either $rdef\big(G/\langle\langle W\rangle\rangle\big)>1$, or $G/\langle\langle W\rangle\rangle$ has a finite index subgroup $N$ such that $rdef(N)>1$.
\end{prop}
\begin{proof}
By \cref{prop: rdef qt fin order}
\[
rdef\big(G/\langle\langle W\rangle\rangle\big)\geq rdef(G)-relsize(W;G,R_G).
\]
Since $W$ has $n-1$ elements, then $relsize(W;G,R_G)\leq n-1$. If $relsize(W;G,R_G)<n-1$, then $rdef\big(G/\langle\langle W\rangle\rangle\big)>1$.
\vskip 2mm

Suppose $relsize(W;G,R_G)=n-1$. First, this means that 
\[
rdef(G/\langle\langle W\rangle\rangle)\geq rdef(G)-relsize(W;G,R_G)=1. 
\]
Second, that $\nu(w_i;G,R_G)=1$, for all $i$, $1\leq i\leq n-1$. This implies that $w_i$ has no minimal root in $G$ with respect to $R_G$, with order greater than one in $G/R_G$. In terms of the notation in \cref{thm: res size inequality: alternative}, if $H$ is a finite index normal subgroup of $G$, then $l_i=1$ for all $i$, $1\leq i\leq n-1$.
\vskip 2mm

Consider $a$ the non-trivial element in $C_G(w_i)-R_G$. Since $G/R_G$ is residually finite, then there is a finite index normal subgroup $H$ in $G$, such that $a\notin H$. Therefore, $k_i=|C_G(w_i):C_H(w_i)|>1$. Since $l_i=1$ for all finite index normal subgroups $H$, then $l_i/k_i=1/k_i<1$. Then, by the arguments preceding \cref{thm: resdef equal to one}, $H/\langle\langle W\rangle\rangle=:N$, is a finite index normal subgroup of $G/\langle\langle W\rangle\rangle$ which has residual deficiency greater than one.
\end{proof}
\vskip 2mm

\begin{ex}
In \cite{toledo}, D. Toledo proved the existence of short exact sequences
\[
1\longrightarrow K\longrightarrow \Phi\longrightarrow \Gamma \longrightarrow 1,
\]
with the following properties. The group $\Gamma$ is residually finite and infinite. The group $\Phi$ is finitely presented (the fundamental group of a smooth projective variety) and not residually finite. The injection $K\longrightarrow \Phi$ is central. That is, $K$, as a subgroup of $\Phi$, is contained in $Z(\Phi)$, the centre of $\Phi$. Finally, as $\Phi/K$ is isomorphic to $\Gamma$ and $\Gamma$ is residually finite, then $R_{\Phi}$ is contained in $K$.
\vskip 2mm

In order to apply \cref{thm: resdef equal to one}, we need $rdef(\Phi)\geq 2$. We do not know what is the residual deficiency of $\Phi$, so take $k\in\mathbb{N}$ such that $rdef(\Phi)+k\geq 2$ and consider $G=\Phi\ast F_k$. Therefore, $rdef(G)=rdef(\Phi\ast F_k)\geq rdef(\Phi)+k\geq 2$. Moreover, since $\Phi\leqslant G$, then $G$ is not residually finite.
\vskip 2mm

Note that $R_G\cap \Phi=R_\Phi$. Therefore, if $\phi\in\Phi-R_{\Phi}$ then $\phi\notin R_{G}$. As $\Gamma\cong \Phi/K$ is non-trivial and $R_{\Phi}\leqslant K$, then such $\phi$ exists.

\vskip 2mm

Consider $W=\{w_1,\ldots,w_n\}\in R_{\Phi}$, where $n$ is such that $rdef(\Phi)+k-n=1$. As $R_{\Phi}$ is contained in $Z(\Phi)$, then $w_i$ commutes with all elements of $\Phi$, for all $i$, $1\leq i\leq n$. Consider $\{t_1,\ldots,t_k\}$ a free set of generators for $F_k$. Let $t$ be any of them. Then $tw_i t^{-1}$ commutes with all elements in $t\Phi t^{-1}$, for all $i$, $1\leq i\leq n$.
\vskip 2mm

So far we have $tw_it^{-1}\in t R_{\Phi}t^{-1}\leqslant tZ(\Phi)t^{-1}\leqslant t\Phi t^{-1}\leqslant G$. We also know that $tR_{\Phi}t^{-1}\leqslant R_G$ and that if $\phi\in\Phi-R_{\Phi}$, then not only $\phi\notin R_{G}$ but also $t\phi t^{-1}\notin R_G$ as $R_G$ is normal in $G$. Since $t\phi t^{-1}$ commutes with $tw_it^{-1}$ for all $\phi\in\Phi-R_{\Phi}$ and all $i$, then $t\phi t^{-1} \in C_G(tw_it^{-1})-R_G$ for all $\phi\in\Phi-R_{\Phi}$ and any $i$, $1\leq i\leq n$.
\vskip 2mm

Let $tWt^{-1}=\{tw_1t^{-1},\ldots,tw_nt^{-1}\}$. We already have that if $\phi\in\Phi-R_{\Phi}$, then $t\phi t^{-1} \in C_G(tw_it^{-1})-R_G$. Therefore, by \cref{thm: resdef equal to one}, either the residual deficiency of $G/\langle\langle tWt^{-1} \rangle\rangle$ is greater than one, or $G/\langle\langle tWt^{-1} \rangle\rangle$ has a finite index subgroup with residual deficiency greater than one.
\end{ex}


\vskip 10 mm



\begin{thebibliography}{99}


\bibitem{abert-nikolov} M. Abert, N. Nikolov \emph{Rank gradient, cost of groups and the rank versus Heegard genus problem}, arXiv:math/0701361v3

\bibitem{agol} I. Agol, D. Groves, J. Manning \emph{The virtual Haken conjecture} arXiv:1204.2810v1 [math.GT].


\bibitem{allcock} D. Allcock, \emph{Spotting infinite groups}, Math. Proc. Camb. Phil. Soc 125 (1999), no. 1, 39-42.

\bibitem{puchta-yiftach} Y. Barnea, J.-C. Schlage-Puchta, \emph{On $p$-deficiency in groups}, arXiv:1106.3255v1

\bibitem{baums-pride} B. Baumslag, S.J. Pride, \emph{Groups with two more generator than relators} J. London Math. Soc (2) \textbf{17}, 425-426 (1978).


\bibitem{gen traingle groups} G. Baumslag, J. W. Morgan, P.B. Shalen, \emph{Generalized triangle groups}, Math. Proc. Camb. Phil. Soc (1987) 102, 25-31.

\bibitem{button lerf} J.O. Button, \emph{Largeness of LERF and 1-relator groups}, arXiv:0803.3805


\bibitem{jack-anitha} J.O. Button, A. Thillaisundaram, \emph{Applications of $p$-deficiency and $p$-largeness}, International Journal of Algebra and Computation (IJAC), Volume: 21, Issue: 4(2011) pp. 547-574.

\bibitem{thomson gp notes} J.W. Cannon, W.J. Floyd, W.R. Parry, \emph{Introductory notes on Richard Thomson's groups}, Enseign. Math. 42 (1996) 215-256.

\bibitem{davis} M.W. Davis, \emph{The Geometry and Topology of Coxeter Groups}, London Mathematical Society Monographs Series, Vol. 32.



\bibitem{eckman} B. Eckmann, \emph{Introduction to $l_2$ methods in topology: reduced $l_2$-homology, harmonic chains, $l_2$-Betti numbers}, Notes prepared by Guido Mislin. Israel J. Math. 117 (2000), 183-219.

\bibitem{Edjvet-balanced} M. Edjvet, \emph{Groups with balanced presentations}, Arch. Math., Vol. 42, 311-313 (1984).

\bibitem{epstein} D.B.A Epstein, \emph{Finite presentations of group and $3$-manifolds}, Q. J. Math. (1961) 12(1): 205-212. 

\bibitem{Fine-Howie-Rosenberger} B. Fine, J. Howie, G. Rosenberger, \emph{One-Relator Quotients and Free Products of Cyclics}, Proc. Amer. Math. Soc. Vol. 102, No.2 pp.249-254.

\bibitem{fine levin rosenberger} B. Fine, F. Levin, F. Roehl and G. Rosenberger, \emph{The generalized tetrahedron groups}, in
R. Charney, M. Davis and M. Shapiro (eds.), Geometric Group Theory, Ohio State University,
Mathematical Research Institute Publications 3, de Gruyter (1995), 99-119.

\bibitem{harpe} P. de la Harpe, \emph{Topics in geometric group theory}, Chicago Lectures in Mathematics, University of Chicago Press, Chicago, Il, 2000.


\bibitem{kar niblo} A. Kar, G.A. Niblo, \emph{Some non-amenable groups}, Publ. Mat. Volume 56, Number 1 (2011), 255-259.

\bibitem{lackers covers of 3 orbs} M. Lackenby, \emph{Covering spaces of 3-orbifolds}, Duke Math. J., 136 (2007), 181-203.

\bibitem{lackers RG} M. Lackenby, \emph{Expanders, rank and graphs of groups}, Israel J. Math. 146 (2005), 357-370.

\bibitem{lackers hom growth of subgps} M. Lackenby, \emph{Large groups, Property (tau) and the homology growth of subgroups}, Math. Proc. Camb. Phil. Soc. 146 (2009) 625-648.


\bibitem{Lackers 3 orbs} M. Lackenby \emph{Some 3-manifolds and 3-orbifolds with large fundamental group} Proc. Amer. Math. Soc. 135 (2007) 3393-3402.


\bibitem{Lubotzky} A. Lubotzky, \emph{Discrete Groups, Expanding Graphs and Invariant Measures}, Progr.
in Math. 125 (1994).

\bibitem{lubotzky} A. Lubotzky, \emph{Free Quotients and the first Betti number of some hyperbolic manifolds}, Transform. Groups 1. (1996), 71-82.

\bibitem{lubotzky segal} A. Lubotzky, D. Segal, Subgroup Growth, Progress in Mathematics 212, Birkha\"{u}ser Verlag, Bassel, 2003.

\bibitem{luck approx thm} W. L\"{u}ck, \emph{Approximating $L^2$-invariants by their finite-dimensional analogues}, Geom. Funct. Anal. 4 (1994), no. 4
, 455-481.

\bibitem{luck} W. L\"{u}ck, D. Osin, \emph{Approximating the first $L^2$-betti number of residually finite groups}, Journal of Topology and Analysis (JTA), Volume: 3, Issue: 2(2011) pp. 153-160.

\bibitem{lyndon} R.C Lyndon, P. E. Schupp, \emph{Combinatorial Group Theory}, Reprint of the 1977 edition. Classics in Mathematics. Springer-Verlag, Berlin, 2001.

\bibitem{malcev} A.I. Mal'cev, \emph{On faithful representations of infinite groups of matrices}, American Math. Soc. Translations (2) 45 (1965), 1-18.

\bibitem{neumann} P.M. Neumann, \emph{The SQ-universality of some finitely presented groups}, J. Austral. Math. Soc (16),(1973), 1-6.

\bibitem{olshanskii-osin} A. Yu. Olshanskii, D. V. Osin, \emph{Large groups and their periodic quotients}, Proc. Amer. Math. Soc. 136 (2008), 753-759.


\bibitem{pride large} S.J. Pride, \emph{The concept of "largeness" in group theory}, in Word problems (II), Stud. Logic Foundations Math. 95, North-Holland, Amsterdam-New York, 1980, pp. 299-335.

\bibitem{reznikov} A. Reznikov, \emph{Volumes of discrete groups and topological complexity of homology spheres}, Math. Ann. 306, 547-554 (1996).


\bibitem{puchta} J.-C. Schlage-Puchta, \emph{A $p$-group with positive rank gradient},  arXiv:1105.1631v1

\bibitem{selberg} A. Selberg, \emph{On discontinuous groups in higher dimensional symmetric spaces}, Int. colloquium on Functional Theory, pages 147-160. Tata Institute, Bombay, 1960.

\bibitem{Stohr} R. Stohr, \emph{Groups with One More Generator than Relators}, Math. Z. 182, 45-47 (1983).

\bibitem{thomas} R. Thomas, \emph{Cayley graphs and group presentations}, Math. Proc. Camb. Phil. Soc (1988) 103, 385.

\bibitem{toledo} D. Toledo, \emph{Projective varieties with non-residually finite fundamental group}, Publications mathématiques de l’I.H.É.S., tome 77 (1993), p. 103-119.


\bibitem{wall} C.T.C. Wall, \emph{Homological Group Theory}, London Mathematical Society Lecture Note Series (36), Cambridge University Press, 1979.


\end{thebibliography}
\end{document}